%% file: LuminyNonSMF.tex
\documentclass[english]{amsart}
\usepackage{etex}
\usepackage{amssymb}
\usepackage{cite}
\usepackage{booktabs}
\usepackage{url}
\usepackage{hyphenat}
\usepackage{mathtools}
\usepackage[all]{xypic}
\usepackage{ifpdf}
\usepackage{hyperref}
\usepackage{color}
\usepackage[pdftex,lmargin=1in, rmargin=1in, tmargin=1in, bmargin=1in]{geometry}
\usepackage{graphicx}
\usepackage{todonotes}
\makeatletter
\providecommand\@dotsep{5}
\def\listtodoname{List of Todos}
\def\listoftodos{\@starttoc{tdo}\listtodoname}
\makeatother

\usepackage{fancyhdr}
\newcommand{\thesec}{Introduction}
\newcommand{\sectitle}{}
\input{macros}
\input{defs}

\begin{document}
\title[Heegaard Floer Homologies]{Heegaard Floer Homologies\\ Lecture Notes}

\author{Robert Lipshitz}
 \address{Department of Mathematics, Columbia University\\
   New York, NY 10027}
\date{June 23--27, 2014}
%\email{lipshitz@math.columbia.edu}

\maketitle
\tableofcontents

\newcommand{\correctionhere}{\todo{Fixed}}

\input{LNbody.tex}

\renewcommand{\thesec}{Bibliography}
\renewcommand{\sectitle}{}
\bibliographystyle{hamsalpha}
\bibliography{heegaardfloer}
\end{document}

%% file: macros.tex
\newread\testin

\graphicspath{{draws/}{mpdraws/}{}}
\makeatletter
\def\input@path{{}{draws/}}
\makeatother

\def\mathcenter#1{%
  \vcenter{\hbox{$#1$}}%
}

\makeatletter
\newcommand\mi@kern[1]{%
  \settowidth\@tempdima{$\mi@obj^{#1}$}
  \kern-\@tempdima
  #1
  \settowidth\@tempdima{$\mi@obj$}
  \kern\@tempdima
}

\newtoks\mi@toksp
\newtoks\mi@toksb
\DeclareRobustCommand{\manyindices}[5]{
  \def\mi@obj{#5}
  \mi@toksp\expandafter{\mi@kern{#2}}
  \mi@toksb\expandafter{\mi@kern{#1}}
  \@mathmeasure4\textstyle{#5_{#1}^{#2}}
  \@mathmeasure6\textstyle{#5_{#3}^{#4}}
  \dimen0-\wd6 \advance\dimen0\wd4
  \@mathmeasure8\textstyle{\hphantom{{}_{#1}^{#2}}#5^{\the\mi@toksp#4}_{\the\mi@toksb#3}}
  \hbox to \dimen0{}{\kern-\dimen0\box8}
}
\makeatother 

\usepackage{ifpdf}
\ifpdf
  
\else
  
\fi

%% file: defs.tex
\newcommand{\RR}{\mathbb R}
\newcommand{\CC}{\mathbb C}
\newcommand{\ZZ}{\mathbb Z}
\newcommand{\QQ}{\mathbb Q}

\newcommand{\FF}{\mathbb F}
\newcommand{\NN}{\mathbb N}

\newcommand{\bD}{\mathbb{D}}

\newcommand{\co}{\nobreak\mskip2mu\mathpunct{}\nonscript
  \mkern-\thinmuskip{:}\penalty300\mskip6muplus1mu\relax}

\newcommand{\bdy}{\partial}

\newcommand{\lbracket}{[}
\newcommand{\rbracket}{]}

\newcommand{\spinc}{\mathfrak s}

\DeclareMathOperator{\Sym}{Sym}

\DeclareMathOperator{\rank}{rank}

\DeclareMathOperator{\spin}{spin}
\newcommand{\SpinC}{\spin^c}

\DeclareMathOperator{\ind}{ind}

\DeclareMathOperator{\Cone}{Cone}

\newcommand\interior{\mathrm{int}}

\DeclareMathOperator{\SO}{\mathit{SO}}

\theoremstyle{plain}

\numberwithin{equation}{section}
\newtheorem{theorem}[equation]{Theorem}

\newtheorem{proposition}[equation]{Proposition}
\newtheorem{lemma}[equation]{Lemma}
\newtheorem{corollary}[equation]{Corollary}

\newtheorem{definition}[equation]{Definition}

\theoremstyle{definition}

\theoremstyle{remark}
\newtheorem{example}[equation]{Example}
\newtheorem{remark}[equation]{Remark}
\newcommand{\Kh}{\mathit{Kh}}
\newcommand{\rKh}{\widetilde{\Kh}}
\newcommand{\HF}{\mathit{HF}}

\newcommand{\HFa}{\widehat {\HF}}
\newcommand{\SFH}{\mathit{SFH}}
\newcommand{\SFC}{\mathit{SFC}}

\newcommand{\CF}{{\mathit{CF}}}
\newcommand{\CFa}{\widehat {\mathit{CF}}}

\newcommand{\HFKa}{\widehat{\mathit{HFK}}}
\newcommand{\HFLa}{\widehat{\mathit{HFL}}}
\newcommand{\x}{\mathbf x}
\newcommand{\y}{\mathbf y}
\newcommand{\z}{\mathbf z}

\newcommand\HH{\mathit{HH}}

\newcommand\Hochschild\HH

\newcommand{\alphas}{{\boldsymbol{\alpha}}}
\newcommand{\betas}{{\boldsymbol{\beta}}}

\newcommand{\cM}{\mathcal{M}}

\newcommand{\HFK}{\mathit{HFK}}

\newcommand\Id{\mathbb{I}}

\newcommand{\Field}{{\FF_2}}
\DeclareMathOperator{\nbd}{nbd}

\newcommand{\Heegaard}{\mathcal{H}}
\newcommand{\HD}{\Heegaard}

\renewcommand{\th}{^\text{th}}

\newcommand{\HB}{\mathsf{H}}

\makeatletter
\newcommand\honestalg[3]{\bigl\lbracket
\begin{smallmatrix} #1\@ifempty{#3}{}{&#3} \\ #2 \end{smallmatrix}
\bigr\rbracket}
\makeatother
\newcommand{\pt}{\mathit{pt}}

\newcommand{\Crit}{\mathit{Crit}}
\newcommand{\Foli}{\mathcal{F}}

%% file: LNbody.tex
\section{Introduction and overview}

\subsection{A brief overview}
Heegaard Floer homology is a family of invariants of objects
in low-dimensional topology. The first of these invariants were
introduced by Ozsv\'ath-Szab\'o: invariants of closed
3-manifolds and smooth 4-dimensional
cobordisms~\cite{OS04:HolomorphicDisks,OS06:HolDiskFour} (see also~\cite{JT:Naturality}). Later, Ozsv\'ath-Szab\'o and,
independently, Rasmussen introduced invariants of knots in
3-manifolds~\cite{OS04:Knots,Rasmussen03:Knots}. There are also
several other invariants, including invariants of contact structures~\cite{OS05:Contact},
more invariants of knots and 3-manifolds~\cite{BrDCov,AbsGraded}, and invariants of Legendrian
and transverse knots~\cite{OST,HKM09:Sutured,LOSS09:transverse}. The subject has had many applications; I will
not even try to list them here, though we will see a few in the
lectures.

In the first three of these lectures, we will focus on a
generalization of one variant of these invariants: an invariant of
sutured 3-manifolds, due to Juh\'asz, called \emph{sutured Floer
  homology}~\cite{Juhasz06:Sutured}. The main goal will be to relate
these invariants to ideas in more classical 3-manifold topology. In
particular, we will sketch a proof that sutured Floer homology detects
the genus of a knot. The proof, which is due to Juh\'asz~\cite{Juhasz08:SuturedDecomp} extending earlier results of Ozsv\'ath-Szab\'o~\cite{OS04:GenusBounds}, uses Gabai's theory of sutured manifolds and sutured hierarchies, which we will review in the first lecture.

In the fourth lecture, we go in a different direction: we will talk
about the surgery exact sequence in Heegaard Floer homology. The goal
is to sketch a (much studied) relationship between Heegaard Floer
homology and Khovanov homology: a spectral sequence due to
Ozsv\'ath-Szab\'o~\cite{BrDCov}.

\subsection{A more precise overview}
Heegaard Floer homology assigns to each closed, oriented, connected
3-manifold $Y$ an abelian group $\HFa(Y)$, and $\ZZ[U]$-modules
$\HF^+(Y)$, $\HF^-(Y)$ and $\HF^\infty(Y)$. These are the homologies
of chain complexes $\CFa(Y)$, $\CF^+(Y)$, $\CF^-(Y)$ and
$\CF^\infty(Y)$. These chain complexes are related by short exact sequences
\[
0 \longrightarrow \CF^-(Y) \stackrel{\cdot U}{\longrightarrow} \CF^\infty(Y) \longrightarrow \CF^+(Y) \longrightarrow 0
\]
\[
  0 \longrightarrow \CF^-(Y) \stackrel{\cdot U}{\longrightarrow} \CF^-(Y) \longrightarrow \CFa(Y) \longrightarrow 0
\]
\[
  0 \longrightarrow \CFa(Y) \longrightarrow \CF^+(Y) \stackrel{\cdot U}{\longrightarrow} \CF^+(Y) \longrightarrow 0  
\]
which, of course, induce long exact sequences in homology. In
particular, either of $\CF^+(Y)$ or $\CF^-(Y)$ determines
$\CFa(Y)$. (The complexes $\CF^+(Y)$ and $\CF^-(Y)$ also have
equivalent information, though this does not quite follow from what
we have said so far.) These invariants are defined
in~\cite{OS04:HolomorphicDisks}. (Some students report finding it
helpful to read\cite{Lipshitz06:CylindricalHF} in conjunction
with~\cite{OS04:HolomorphicDisks}.) It is now known, by work of
Hutchings~\cite{Hutchings02:ECH}, Hutchings-Taubes~\cite{HutchingsTaubes07:GluingI,HutchingsTaubes09:GluingII}, Taubes\cite{Taubes10:ECH-SW1,Taubes10:ECH-SW2,Taubes10:ECH-SW3,Taubes10:ECH-SW4,Taubes10:ECH-SW5}, and Kutluhan-Lee-Taubes\cite{KutluhanLeeTaubes:HFHMI,KutluhanLeeTaubes:HFHMII,KutluhanLeeTaubes:HFHMIII,KutluhanLeeTaubes:HFHMIV,KutluhanLeeTaubes:HFHMV} or Colin-Ghiggini-Honda\cite{ColinGhigginiHonda11:HF-ECH-1,ColinGhigginiHonda11:HF-ECH-2,ColinGhigginiHonda11:HF-ECH-3},
that these invariants correspond to different variants of Kronheimer-Mrowka's Seiberg-Witten Floer
homology groups~\cite{KronheimerMrowka}.

Roughly, smooth, compact, connected 4-dimensional cobordisms between
connected 3-manifolds induce chain maps on $\CFa$, $\CF^\pm$ and
$\CF^\infty$, and composition of cobordisms corresponds to composition
of maps. From the maps on $\CF^\pm$ and the exact sequences above, one
can recover the Seiberg-Witten invariant, or at least something very
much like it~\cite{OS06:HolDiskFour}. Note, in particular,
that $\CFa$ does not have enough information to recover the
Seiberg-Witten invariant.

There is an extension of the Heegaard Floer homology groups to nullhomologous knots
in 3-manifolds, called \emph{knot Floer
  homology}~\cite{OS04:Knots,Rasmussen03:Knots}. Given a knot $K$ in a
$3$-manifold $Y$ there is an induced filtration of $\CFa(Y)$,
$\CF^+(Y)$, and so on. In particular, we can define the \emph{knot
  Floer homology groups} $\HFKa(Y,K)$, the homology of the associated
graded complex to the filtration on $\CFa(Y)$. (So, there is a
spectral sequence from $\HFKa(Y,K)$ to $\HFa(Y)$.)

The gradings in the subject are quite subtle. The chain complexes
$\CFa(Y)$, $\CF^+(Y)$, \dots, decompose as direct sums according to
\emph{$\SpinC$-structures} on $Y$, i.e.,
\[
\CFa(Y)=\bigoplus_{\spinc\in\SpinC(Y)}\CFa(Y,\spinc).
\]
(We will discuss $\SpinC$ structures more in Section~\ref{sec:spinc}.)
Each of the $\CFa(Y,\spinc)$ is relatively graded by some $\ZZ/n\ZZ$
(where $n$ is the divisibility of $c_1(\spinc)$). In particular, if
$c_1(\spinc)=0$ (i.e., $\spinc$ is \emph{torsion}) then
$\CFa(Y,\spinc)$ has a relative $\ZZ$ grading. Similarly, $\HFKa(Y,K)$
decomposes as a direct sum of groups, one per relative $\SpinC$
structure on $(Y,K)$.

In the special case that $Y=S^3$, there is a canonical identification
$\SpinC(S^3,K)\cong \ZZ$, and each $\HFKa(Y,K,\spinc)$ in fact has an
absolute $\ZZ$-grading. That is, $\HFKa(Y,K)$ is a bigraded abelian
group. We will write
$\HFKa(S^3,K)=\HFKa(K)=\bigoplus_{i,j}\HFKa_i(K,j)$, where $j$ stands
for the $\SpinC$ grading. The grading $j$ is also called the
\emph{Alexander grading}, because
\[
\sum_{i,j}(-1)^it^j\rank\HFKa_i(K,j)=\Delta_K(t),
\]
the (Conway normalized) Alexander polynomial of $K$.

The breadth of the Alexander polynomial $\Delta_K(t)$, or equivalently
the degree of the symmetrized Alexander polynomial, gives a lower
bound on the genus $g(K)$ of $K$ (i.e., the minimal genus of any Seifert
surface for $K$). One of the main goals of these lectures will be to
sketch a proof of the following refinement:
\begin{theorem}\cite{OS04:GenusBounds}\label{thm:knot-genus}
  Given a knot $K$ in $S^3$,
  $
  g(K)=\max\{j\mid \bigl(\bigoplus_i \HFKa_i(K,j)\bigr)\neq 0\}.
  $
\end{theorem}

Rather than giving the original proof of Theorem~\ref{thm:knot-genus},
we will give a proof using an extension of $\HFa$ and $\HFKa$, due to
Juh\'asz, called \emph{sutured Floer homology}. Sutured manifolds were
introduced by Gabai in his work on foliations, fibrations, the
Thurston norm, and knot
genus~\cite{Gabai83:foliations,Gabai84:knot-foliations,Gabai86:arborescent,Gabai87:foliations-2-3};
we will review some aspects of this theory in the first lecture.
Sutured Floer homology associates to each sutured manifold
$(Y,\Gamma)$ satisfying certain conditions (called being
\emph{balanced}) a chain complex $\SFC(Y,\Gamma)$ whose homology
$\SFH(Y,\Gamma)$ is an invariant of the sutured manifold. These chain
complexes behave in a particular way under Gabai's \emph{surface
  decompositions}, leading to a proof of Theorem~\ref{thm:knot-genus}.

In the last lecture, we turn to a different topic: the behavior of
Heegaard Floer homology under knot surgery. The goal is to relate
these lectures to the lecture series on Khovanov homology. In
particular, we will sketch the origins of Ozsv\'ath-Szab\'o's spectral
sequence $\rKh(m(K))\Rightarrow \HFa(\Sigma(K))$ from the (reduced)
Khovanov homology of the mirror of $K$ to the Heegaard Floer homology
of the branched double cover of $K$~\cite{BrDCov}.

\subsection{References for further reading}
There are a number of survey articles on Heegaard Floer homology. Three
by Ozsv\'ath-Szab\'o~\cite{OS05:EMS-survey,OS06:Clay1,OS06:Clay2} give
nice introductions to the Heegaard Floer invariants of 3- and
4-manifolds and knots. Juh\'asz's recent survey~\cite{Juhasz:Survey}
contains an introduction to sutured Floer homology, which is the main
subject of these lectures. There are also some more focused surveys of
other recent developments~\cite{Manolescu:knot-survey,LOT:tour}.

Sutured Floer homology, as we will discuss it, is developed in a pair
of papers by
Juh\'asz~\cite{Juhasz06:Sutured,Juhasz08:SuturedDecomp}. For a
somewhat different approach to relating sutured manifolds and Floer
theory, see the work of Ni (starting perhaps with~\cite{Ni09:FiberedMfld}).

\subsection*{Acknowledgments} I thank the participants and organizers
of the 2014 SMF summer school ``Geometric and Quantum Topology in
Dimension 3'' for many corrections to an earlier draft. I also thank
Zhechi Cheng, Andr\'as Juh\'asz, Peter Ozsv\'ath, Dylan Thurston, and Mike Wong for
further suggestions and corrections. I was partly supported by NSF Grant DMS-1149800, and partly by the Soci\'et\'e Mathematique de France.

\section{Sutured manifolds, foliations and sutured hierarchies}
\renewcommand{\thesec}{Lecture 1}
\renewcommand{\sectitle}{Sutured manifolds, foliations and sutured hierarchies}
\subsection{The Thurston norm and foliations}
\begin{definition}
  Given a knot $K\subset S^3$, the \emph{genus of $K$} is the minimal
  genus of any Seifert surface for $K$ (i.e., of any embedded surface
  $F\subset S^3$ with $\bdy F=K$).
\end{definition}

Thurston found a useful generalization of this notion to arbitrary
3-manifolds and, more generally, to link complements in arbitrary
3-manifolds:
\begin{definition}
  For a $3$-manifold $Y$ with boundary $\bdy Y$ a disjoint union of
  tori, the \emph{Thurston norm}
  \[
  x\co H_2(Y,\bdy Y)\to \ZZ
  \]
  is defined as follows. Given a compact, oriented surface $F$ (not
  necessarily connected, possibly with boundary) define the
  \emph{complexity of $F$} to be
  \[
  x(F)=\sum_{\chi(F_i)\leq 0} |\chi(F_i)|,
  \]
  where the sum is over the connected components $F_i$ of $F$.
  
  Given an element $h\in H_2(Y,\bdy Y)$ and a surface $F\subset
  Y$ with $\bdy F\subset \bdy Y$ we say that \emph{$F$ represents $h$}
  if the inclusion map sends the fundamental class of $F$ in
  $H_2(F,\bdy F)$ to $h$. Define
  \[
  x(h)=\min\{x(F)\mid F\text{ an embedded surface representing }h\}.
  \]
\end{definition}

For this definition to make sense, we need to know the surface $F$
exists:
\begin{lemma}
  Any element $h\in H_2(Y,\bdy Y)$ is represented by some embedded surface $F$.
\end{lemma}
\begin{proof}[Idea of Proof]
  The class $h$ is Poincar\'e dual to a class in $H^1(Y)$, which in
  turn is represented by a map $f_h\co Y\to K(\ZZ,1)=S^1$. The
  preimage of a regular value of $f_h$ represents $h$.
  See for instance~\cite[Lemma 1]{Thurston86:norm} for more details.
\end{proof}

\begin{proposition}
  If $(Y,\bdy Y)$ has no essential spheres ($Y$ is \emph{irreducible}) or
  disks ($\bdy Y$ is \emph{incompressible}) then $x$ defines a
  pseudo-norm on $H_2(Y,\bdy Y)$ (i.e., a norm except for the
  non-degeneracy axiom). If moreover $Y$ has no essential annuli or
  tori ($Y$ is \emph{atoroidal}) then $x$ defines a norm on $H_2(Y,\bdy Y)$,
  and induces a norm on $H_2(Y,\bdy Y;\QQ)$.
\end{proposition}
\begin{proof}[Idea of Proof]
  The main points to
  check are that:
  \begin{enumerate}
  \item $x(n\cdot h)=n\cdot x(h)$ for $n\in\NN$. 
  \item $x(h+k)\leq x(h)+x(k)$.
  \end{enumerate}
  For the first point, a little argument shows that a surface
  representing $n\cdot h$ (with $h$ indivisible) necessarily has $n$
  connected components, each representing $h$. The second is a little
  more complicated. Roughly, one takes surfaces representing $h$ and
  $k$ and does surgery on their circles and arcs of intersection to
  get a new surface representing $h+k$ without changing the Euler
  characteristic. (More precisely, one first has to eliminate
  intersections which are inessential on both surfaces, as doing
  surgery along them would create disjoint $S^2$ or $\bD^2$
  components.) See~\cite[Theorem 1]{Thurston86:norm} for details. 
\end{proof}

\begin{example}
  If $Y=S^3\setminus\nbd(K)$ is the exterior of a knot then
  $H_2(Y,\bdy Y)\cong \ZZ$ and surfaces representing a generator for
  $H_2(Y,\bdy Y)$ are Seifert surfaces for $K$. The Thurston norm of a
  generator is given by $2g(K)-1$ (if $K$ is not the unknot).
\end{example}

\begin{example}
  Consider $Y=S^1\times \Sigma_g$, for any $g>0$. Fix a collection of curves
  $\gamma_i$, $i=1,\dots,2g$, in $\Sigma$ giving a basis for
  $H_1(\Sigma)$. Then $H_2(Y)\cong \ZZ^{2g+1}$, with basis (the
  homology classes represented by) $S^1\times \gamma_i$,
  $i=1,\dots,2g$, and $\{\pt\}\times\Sigma$. We have $x([S^1\times
  \gamma_i])=0$, from which it follows (why?) that $x$ is determined
  by $x([\{pt\}\times\Sigma])$. One can show using elementary
  algebraic topology that $x([\{pt\}\times\Sigma])=2g-2$; see
  Exercise~\ref{exercise:S1-times-Sigma}.
\end{example}

\begin{remark}
  A norm is determined by its unit ball. The Thurston norm ball turns
  out to be a polytope defined by inequalities with integer
  coefficients~\cite[Theorem 2]{Thurston86:norm}.
\end{remark}

\emph{A priori}, the Thurston norm looks impossible to compute in
general. Remarkably, however, it can be understood. The two key
ingredients are foliations, which we discuss now, and a decomposition
technique, due to Gabai, which we discuss next.
\begin{definition}
  A \emph{smooth, codimension-1 foliation $\Foli$} of $M$ is a
  collection of disjoint, codimension-1 immersed submanifolds
  $\{N_j\subset M\}_{j\in J}$ so that each immersion is injective and
  for any $x\in M$ there is a neighborhood $U\ni x$ and a
  diffeomorphism $\phi\co U\to \RR^n$ so that for each $t\in\RR$,
  $\phi^{-1}(\RR^{n-1}\times\{t\})\subset N_j$ for some $j=j(t)$.
  The $N_j$ are called the \emph{leaves} of the foliation.
\end{definition}

We will only be interested in smooth, codimension-1 foliations, so we
will refer to these simply as foliations. (Actually, there are good
reasons to consider non-smooth foliations in this setting. Higher
codimension foliations are also, of course, interesting.)

In a small enough neighborhood of any point, the $N_j$ look like pages
of a book, though each $N_j$ may correspond to many pages. The
standard examples are foliations of the torus
$T^2=[0,1]\times[0,1]/\sim$ by the curves $\{y=m x+b\}$ for fixed
$m\in\RR$ and $b$ allowed to vary. If $m$ is rational then the leaves
are circles. If $m$ is irrational then the leaves are immersed copies
of $\RR$.

The tangent spaces to the leaves $N_j$ in a foliation $\Foli$ of $M^n$
define an $(n-1)$-plane field in $TM$. This is the \emph{tangent
space to $\Foli$}, which we will write as $T\Foli$. An \emph{orientation} of
$\Foli$ is an orientation of $T\Foli$, and a \emph{co-orientation} is an
orientation of the orthogonal complement $T\Foli^\perp$ of
$T\Foli$. Since we are only interested in oriented $3$-manifolds, the
two notions are equivalent.

A curve is \emph{transverse to $\Foli$} if it is transverse to $T\Foli$.

\begin{definition}
  A foliation $\Foli$ of $M$ is called \emph{taut} if there is a curve
  $\gamma$ transverse to $\Foli$ such that $\gamma$ intersects every
  leaf of $\Foli$.
\end{definition}

\begin{theorem}\label{thm:norm-and-foli}\cite[Corollary 2, p. 119]{Thurston86:norm}
  Let $\Foli$ be a taut foliation of $Y$ so that for every component $T$ of
  $\bdy Y$ either:
  \begin{itemize}
  \item $T$ is a leaf of $\Foli$ or
  \item $T$ is transverse to $\Foli$ and $\Foli\cap T$ is taut in $T$.
  \end{itemize}
  Then every compact leaf of $\Foli$ is genus minimizing.
\end{theorem}
(The proof is not so easy.)

As we discuss next, Gabai showed that, at least in principle,
Theorem~\ref{thm:norm-and-foli} can always be used to determine the
Thurston norm.

\subsection{Sutured manifolds}
\begin{definition}
  A \emph{sutured manifold} is an oriented $3$-manifold $Y$ together with a
  decomposition of $\bdy Y$ into three parts (codimension-0
  submanifolds with boundary): the bottom part $R_-$, the top part
  $R_+$, and the vertical part $\gamma$. This decomposition is
  required to satisfy the properties that:
  \begin{enumerate}
  \item Every component of $\gamma$ is either an annulus or a torus.
  \item $\bdy R_+\cap \bdy R_-=\emptyset$ (so $\bdy \gamma=\bdy
    R_+\amalg \bdy R_-$).
  \item Each annulus in $\gamma$ shares one boundary component with
    $R_+$ and one boundary component with $R_-$.
  % \item Orient $R_+$ (respectively $R_-$) using the orientation of $Y$
  %   and the outward-pointing (respectively inward-pointing) normal
  %   vector. That is, the orientation of $R_+$ agrees with the standard
  %   orientation of $\bdy Y$, and the orientation of $R_-$ agrees with
  %   $-\bdy Y$. Then both $R_+$ and $R_-$ induce orientations of the
  %   cores of the annuli in $\gamma$, and we require that these
  %   orientations agree.
  \end{enumerate}
  (See~\cite[Section 2]{Gabai83:foliations}.)

  Let $T(\gamma)$ denote the union of the toroidal components of
  $\gamma$ and $A(\gamma)$ the union of the annular components of
  $\gamma$. We will often denote a sutured manifold by $(Y,\gamma)$.

  In a sutured manifold, we give $R_+$ the (outward-normal-first)
  orientation induced on $\bdy Y$ and $R_-$ the orientation as $-\bdy
  Y$, i.e., the opposite of the boundary orientation, so both $\bdy
  R_+$ and $\bdy R_-$ induce the same orientations on the cores of the
  annuli.

  A sutured manifold is called \emph{taut} if $Y$ is irreducible
  (every $S^2$ bounds a $\bD^3$) and $R_+$ and $R_-$ are
  Thurston norm-minimizing in their homology classes (in $H_2(Y,\gamma)$).
  % (Maybe the boundary
  % should also be incompressible; this is perhaps part of irreducible
  % for a 3-manifold with boundary.)

  A sutured manifold is called \emph{balanced} if:
  \begin{enumerate}
  \item $T(\gamma)=\emptyset$.
  \item $R_+$ and $R_-$ have no closed components.
  \item $Y$ has no closed components.
  \item $\chi(R_+)=\chi(R_-)$ for each connected component of $Y$.
  \end{enumerate}
  Let $\Gamma$ denote the cores of the annular components of
  $\gamma$. Orient $\Gamma$ as the boundary of $R_+$. Then for a balanced sutured manifold, $(Y,\Gamma)$
  determines the whole sutured structure, so we may refer to
  $(Y,\Gamma)$ as a sutured manifold.
\end{definition}

\begin{example}
  Given a surface $R$ with boundary, consider $Y=[0,1]\times R$. Make
  this into a sutured manifold by defining $R_+=\{1\}\times R$,
  $R_-=\{0\}\times R$ and $\gamma=[0,1]\times\bdy R$. Sutured
  manifolds of this form are called \emph{product sutured manifolds}.

  Product sutured manifolds are taut and, if $R$ has no closed
  components, balanced.
\end{example}

\begin{example}
  Let $Y$ be a closed, connected 3-manifold. We can view $Y$ as a
  somewhat trivial example of a sutured 3-manifold. This sutured
  3-manifold may or may not be taut, but is not balanced.

  We can also delete a ball $\bD^3$ from $Y$ and place, say, a single
  annular suture on the resulting $S^2$ boundary. (So, $R_+=\bD^2$,
  $R_-=\bD^2$, and $\gamma=[0,1]\times S^1$.) This sutured manifold is
  not taut (unless $Y=S^3$)---a sphere parallel to the boundary does
  not bound a disk---but it is balanced.
\end{example}

\begin{example}\label{eg:sutured-knot-complement}
  Let $Y$ be a closed, connected $3$-manifold and let $K\subset Y$ be
  a knot. Consider $Y\setminus \nbd(K)$, the exterior of $K$. We can
  view this as a sutured manifold by defining $\gamma$ to be the whole
  torus boundary. This sutured manifold is not balanced.

  More relevant to our later constructions, we can define a balanced
  sutured manifold by letting $\Gamma$ consist of $2n$ meridional
  circles, so $R_+$ and $R_-$ each consists of $n$ annuli.  See
  Figure~\ref{fig:sutured-decomp}.  (In my head, this looks like a
  knotted sea monster biting its own tail: $R_+$ is the part above the
  water.)
\end{example}

\begin{definition}
  A foliation $\Foli$ on $Y$ is \emph{compatible with $\gamma$} if
  \begin{enumerate}
  \item $\Foli$ is transverse to $\gamma$.
  \item $R_+$ and $R_-$ are unions of leaves of $\Foli$, and the
    orientations of these leaves agree with the orientations of $R_\pm$.
  \end{enumerate}
  (I think ``compatible'' is not a standard term.)

  A foliation $\Foli$ on $(Y,\gamma)$ is \emph{taut} if
  \begin{enumerate}
  \item $\Foli$ is compatible with $\gamma$.
  \item $\Foli$ is taut.
  \item For each component $S$ of $\gamma$, $\Foli\cap S$ is taut, as
    a foliation of $S$.
  \end{enumerate}
\end{definition}

\begin{example}
  Every product sutured manifold admits an obvious taut foliation,
  where the leaves are $\{t\}\times R$.
\end{example}

\begin{definition}
  We call a sutured manifold \emph{rational homology trivial} or
  \emph{RHT} if the homology group $H_2(Y)$
  vanishes. (This is not a standard term.)
\end{definition}

\begin{example}
  A knot complement in $S^3$ is RHT. If $Y$ is a closed $3$-manifold
  then $Y\setminus \bD^3$ is RHT if and only if $Y$ is a rational
  homology sphere.
\end{example}

\subsection{Surface decompositions and Gabai's theorem}
\begin{definition}\label{def:decompose}\cite[Definition 3.1]{Gabai83:foliations}
  A \emph{decomposing surface} in a sutured manifold $(Y,\gamma)$ is a
  compact, oriented surface with boundary $(S,\bdy S)\subset (Y,\bdy
  Y)$ so that for every component $\lambda$ of $\bdy S\cap \gamma$,
  either:
  \begin{enumerate}
  \item $\lambda$ is a properly embedded, non-separating arc in $\gamma$, or
  \item $\lambda$ is a circle which is essential in the component of
    $\gamma$ containing $\lambda$.
  \end{enumerate}
  We also require that in each torus component $T$ of $\gamma$, the
  orientations of all circles in $S\cap T$ agree, and in each annular
  component $A$ of $\gamma$, the orientation of all circles in $S\cap
  A$ agree with the orientation of the core of $A$.
  
  Given a sutured manifold $(Y,\gamma)$ and a decomposing surface $S$
  we can form a new sutured manifold $(Y',\gamma')$ as
  follows. Topologically, $Y'=Y\setminus\nbd(S)$. Let
  $S_+,S_-\subset \bdy Y'$ denote the positive and negative pushoffs
  of $S$, respectively. Then $R_+' = (R_+\cap \bdy Y')\cup S'_+$
  (minus a neighborhood of its boundary), $R_-'=(R_-\cap \bdy Y')\cup
  S'_-$ (minus a neighborhood of its boundary), and $\gamma'$ is the
  rest of $\bdy Y'$ (cf.\ Exercise~\ref{exercise:gamma-prime}).  We
  call this operation \emph{sutured manifold decomposition} and write
  $(Y,\gamma)\stackrel{S}{\rightsquigarrow}(Y',\gamma')$.
\end{definition}

\begin{figure}
  \centering
  \includegraphics{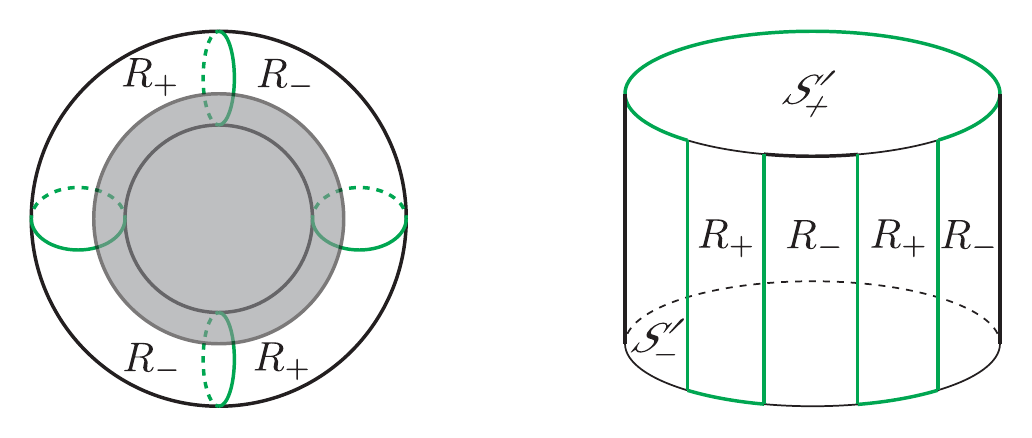}
  \caption{\textbf{A sutured manifold decomposition.} Left: the
    complement of the unknot, with four meridional sutures, together
    with a (gray) decomposing disk. Only the cores of the annular sutures are
    drawn, as green circles. Right: the result of performing a surface
    decomposition to this sutured manifold.}
  \label{fig:sutured-decomp}
\end{figure}

\begin{example}
  If $K$ is a knot in $S^3$, say, $Y=S^3\setminus\nbd(K)$, and
  $\gamma$ consists of $2n$ meridional sutures as in
  Example~\ref{eg:sutured-knot-complement} then any Seifert surface
  for $K$ is a decomposing surface for $(Y,\gamma)$.

  If $K$ is a fibered knot and $F$ is a Seifert surface for $K$ which
  is a fiber of the fibration then the result of doing a surface
  decomposition to the exterior $(Y,\gamma)$ of $K$ is a product
  sutured manifold.  The case that $K$ is the unknot is illustrated in
  Figure~\ref{fig:sutured-decomp}.
\end{example}

It is maybe better to think of the inverse operation to surface
decomposition; perhaps I will say that $(Y,\Gamma)$ is obtained from
$(Y',\Gamma')$ by a \emph{suture-compatible gluing} if $(Y',\Gamma')$
is obtained from $(Y,\Gamma)$ by a surface decomposition. (This is not
a standard term.) Unlike surface decomposition, suture-compatible
gluing is not a well-defined operation: it depends on both a choice of
subsurface $S'\subset \bdy Y'$ and a choice of homeomorphism
$S'_+\cong S'_-$. I think this is why it is not talked about, but for
the behavior of sutured Floer homology gluing seems more natural than
decomposing (especially in view of bordered Floer theory).

\begin{definition}
  We call a decomposing surface $S$ in a balanced sutured manifold
  $(Y,\gamma)$ \emph{balanced-admissible} if $S$ has no closed
  components and for every component $R$ of $R_+$ and $R_-$, the set
  of closed components of $S\cap R$
  is a union of parallel curves (where each of these curves has
  orientation induced by the boundary of $S$), and if these curves are
  null-homotopic then they are oriented as the boundary of their
  interiors.
  (This is not a standard term.)
\end{definition}

\begin{lemma}\label{lem:balanced-to-balanced}
  If $(Y,\Gamma)\stackrel{S}{\rightsquigarrow}(Y',\Gamma')$,
  $(Y,\Gamma)$ is balanced, and $S$ is balanced-admissible then
  $(Y',\Gamma')$ is balanced.
\end{lemma}

The proof is Exercise~\ref{ex:bal-to-bal}.

A particularly simple kind of sutured decomposition is the following:
\begin{definition}
  A \emph{product disk} in $(Y,\Gamma)$ is a decomposing surface $S$
  for $(Y,\Gamma)$ so that $S\cong \bD^2$ and $S\cap \Gamma$ consists
  of two points.  A \emph{product decomposition} is a sutured
  decomposition $(Y,\Gamma)\stackrel{S}{\rightsquigarrow}(Y',\Gamma')$
  where $S$ is a product disk.
\end{definition}

\begin{lemma}\label{lem:glue-foliation}
  Suppose that $(Y,\Gamma)\stackrel{S}{\rightsquigarrow}(Y',\Gamma')$,
  where $\bdy S$ is disjoint from any toroidal sutures of $Y$. Let
  $\Foli'$ be a foliation on $(Y',\Gamma')$ compatible with $\Gamma'$. Then there is an induced
  foliation $\Foli$ on $(Y,\Gamma)$ compatible with $\Gamma$ with the property that $S$ is a
  leaf of $\Foli$.
\end{lemma}
In other words, suture-compatible gluing takes foliations to
foliations with $S$ as a leaf.  This is the easy case in the proof
of~\cite[Theorem 5.1]{Gabai83:foliations}; the proof is
Exercise~\ref{ex:glue-fol}. The harder case, when $\bdy S$ intersects
some toroidal sutures, takes up most of the proof.

\begin{theorem}\cite[Theorems 4.2 and 5.1]{Gabai83:foliations}
  \label{thm:hierarchy}
  Let $(Y,\gamma)$ be a taut sutured manifold. Then there is a
  sequence of surface decompositions
  \[
  (Y,\gamma)=(Y_1,\gamma_1)\stackrel{S_1}{\rightsquigarrow}
  (Y_2,\gamma_2)
  \stackrel{S_2}{\rightsquigarrow}\cdots
  \stackrel{S_{n-1}}{\rightsquigarrow} (Y_n,\gamma_n)
  \]
  so that $(Y_n,\gamma_n)$ is a product sutured manifold, and so that
  moreover there is an induced taut foliation on $(Y,\gamma)$.
\end{theorem}
\begin{proof}[Comments on Proof]
  Gabai's proof of existence of the sequence of decompositions
  (\emph{sutured hierarchy}), Theorem 4.2 in his paper, is an
  intricate induction; even saying what it is an induction on is not
  easy.  Once one has the hierarchy, one uses
  Lemma~\ref{lem:glue-foliation} and its harder cousin for decomposing
  surfaces intersecting $T(\gamma)$ to reassemble the obvious
  foliation of the product sutured manifold $(Y_n,\gamma_n)$ to a
  foliation for $(Y,\gamma)$;  this part is Theorem 5.1 in his paper.
\end{proof}

In fact, Theorem~\ref{thm:hierarchy} has two modest refinements:
\begin{proposition}\label{prop:balanced-hierarchy}(\cite[Theorem 4.19]{Scharlemann89:sutured}, see
  also~\cite[Theorem 8.2]{Juhasz08:SuturedDecomp}) With notation as in
  Theorem~\ref{thm:hierarchy}, if $(Y,\gamma)$ is balanced then we can
  assume the surfaces $S_i$ are all balanced-admissible.
\end{proposition}

\begin{definition}\label{def:good}
  A balanced-admissible decomposing surface $S$ is called \emph{good} if every component of $\bdy S$ intersects both $R_+$ and $R_-$.
  (This is Juh\'asz's term~\cite[Definition 4.6]{Juhasz08:SuturedDecomp}.)
\end{definition}

\begin{proposition}\cite[Lemma 4.5]{Juhasz08:SuturedDecomp}\label{prop:good-hierarchy}
  Any balanced-admissible decomposing surface $S$ is isotopic to a
  good decomposing surface $S'$ so that decomposing along $S$ and
  decomposing along $S'$ give the same result. In particular, in
  Proposition~\ref{prop:balanced-hierarchy}, we can assume the
  decomposing surfaces are all good.
\end{proposition}

\subsection{Suggested exercises}
\begin{enumerate}
\item \label{exercise:S1-times-Sigma} Show, using algebraic topology, that in $S^1\times \Sigma_g$,
  the fiber $\Sigma_g$ is a minimal genus representative of its
  homology class.
\item\label{exercise:gamma-prime} Give an explicit description of $\gamma'$ from Definition~\ref{def:decompose}.
\item Prove that if $(Y,\gamma)\stackrel{S}{\rightsquigarrow}
  (Y',\gamma')$ and $(Y',\gamma')$ is taut then either $Y$ is taut or
  $Y=\bD^2\times S^1$ and $S$ is a disk. (This is~\cite[Lemma
  3.5]{Gabai83:foliations}.)
\item\label{ex:bal-to-bal} Prove Lemma~\ref{lem:balanced-to-balanced}.
\item\label{ex:glue-fol} Prove Lemma~\ref{lem:glue-foliation}.
\end{enumerate}

\section{Heegaard diagrams and holomorphic disks}\label{sec:HF}
\renewcommand{\thesec}{Lecture 2} 
\renewcommand{\sectitle}{Heegaard diagrams and holomorphic disks}

The goal of this lecture is to define sutured Floer homology, compute
some examples of it, and discuss its basic properties. For simplicity,
we will generally work with RHT sutured manifolds, and will always
take coefficients in $\Field=\ZZ/2\ZZ$; neither of these restrictions
present in~\cite{Juhasz06:Sutured}.

Except as noted, the definitions and theorems in this lecture are all
due to Juh\'asz~\cite{Juhasz06:Sutured} (building on earlier work of Ozsv\'ath-Szab\'o,
Rasmussen, and others). Many of the examples predate his work, but I
will state them in his language.

Throughout this lecture, sutured manifold will mean \emph{balanced}
sutured manifold.
\subsection{Heegaard diagrams for sutured manifolds}
\begin{definition}
  A \emph{sutured Heegaard diagram} is a surface $\Sigma$ with
  boundary and tuples $\alphas=\{\alpha_1,\dots,\alpha_n\}$ and
  $\betas=\{\beta_1,\dots,\beta_n\}$ of pairwise disjoint circles in $\Sigma$ so
  that the result $F_-$ (respectively $F_+$) of performing surgery on
  the $\alpha$-circles (respectively $\beta$-circles) has no closed
  components.

  A sutured Heegaard diagram $\HD=(\Sigma,\alphas,\betas)$ specifies a
  sutured $3$-manifold $Y(\HD)$ as follows:
  \begin{itemize}
  \item As a topological space, $Y(\HD)$ is obtained from a thickened
    copy $\Sigma\times[0,1]$ of $\Sigma$ by attaching $3$-dimensional
    $2$-handles along the $\alpha_i\times \{0\}$ and the
    $\beta_i\times\{1\}$.
  \item The boundary of $Y(\HD)$ is $F_-\cup
    \bigl((\bdy\Sigma)\times[0,1]\bigr)\cup F_+$ where $F_-$ and $F_+$ are the parts of $\bdy Y(\HD)$ corresponding to $\Sigma\times\{0\}$ and $\Sigma\times\{1\}$, respectively. We let $R_-=F_-$,
    $R_+=F_+$ and $\Gamma=(\bdy\Sigma)\times \{1/2\}$.
  \end{itemize}
\end{definition}

\begin{figure}
  \centering
  \includegraphics{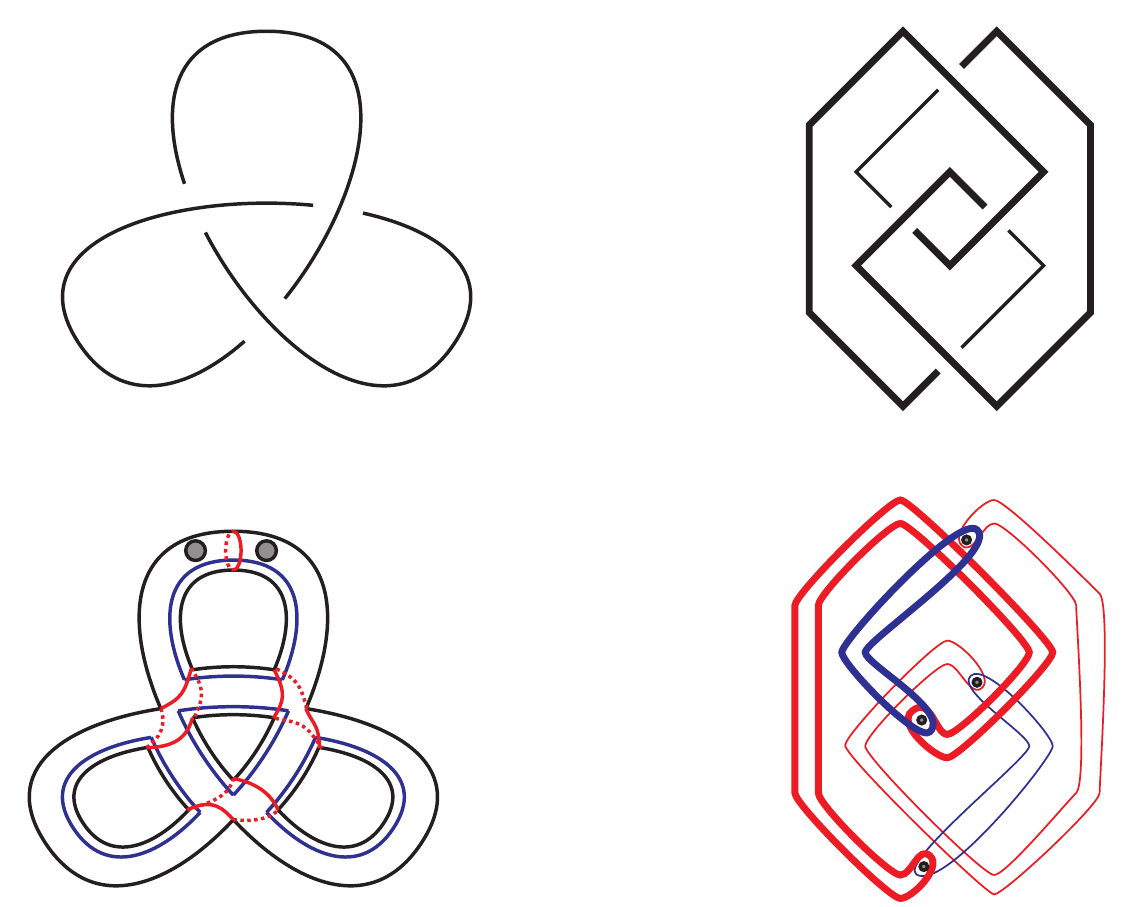}
  \caption{\textbf{Diagrams for knot complements from knot diagrams.}
    Left: the usual diagram for the trefoil and a corresponding
    sutured Heegaard diagram for its exterior. The gray dots are holes
    in the Heegaard surface. Right: a $2$-bridge presentation of the
    trefoil and a corresponding sutured Heegaard diagram. The surface
    $\Sigma$ is $S^2$ minus $4$ disks. In this
    Heegaard diagram, the thin red and blue circles are not part of
    the diagram.}
  \label{fig:knot-to-sutured}
\end{figure}

\begin{example}
  Fix a knot $K\subset S^3$ and a knot diagram $D$ for $K$ with $n$
  crossings. Consider the $3$-manifold $Y=S^3\setminus\nbd(K)$. We can
  find a Heegaard diagram for $Y$ with $2n$ meridional sutures as on
  the left of Figure~\ref{fig:knot-to-sutured}. Alternatively, given
  an $n$-bridge presentation of $K$, there is a corresponding sutured
  Heegaard diagram for $K$ with $2n$ sutures; see the right of
  Figure~\ref{fig:knot-to-sutured}.

  These kinds of Heegaard diagram are exploited
  in~\cite{OSS09:singular,OS09:cube} to give a cube of resolutions
  description of knot Floer homology.
\end{example}

\begin{figure}
  \centering
  \includegraphics{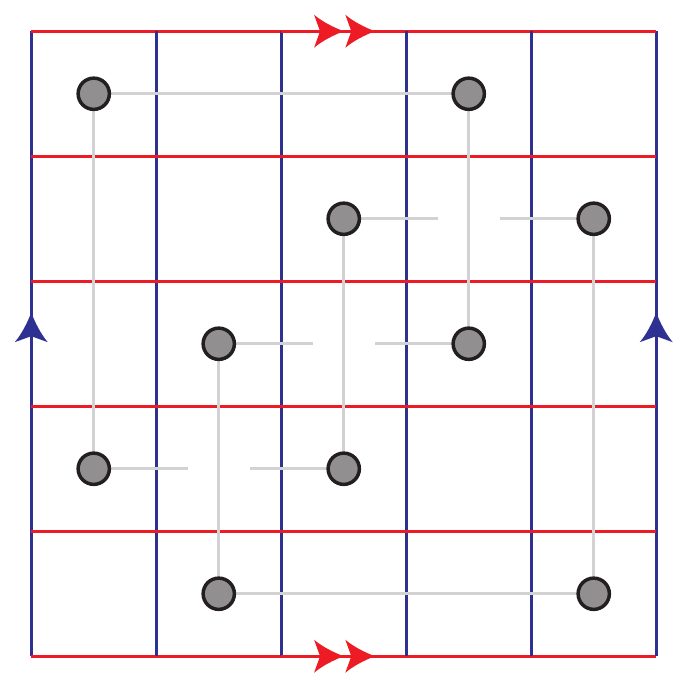}
  \caption{\textbf{A toroidal grid diagram for the trefoil.} The left
    and right edges of the diagram are identified. The knot itself is
    shown in light gray.}
  \label{fig:toroidal-trefoil}
\end{figure}

\begin{example}\label{eg:grid-diags}
  An \emph{$n\times n$ toroidal grid diagram} is a special kind of
  sutured Heegaard diagram in which the $\alpha$-circles (respectively
  $\beta$-circles) are $n$ horizontal (respectively vertical) circles
  on a torus with $2n$ disks removed. Each horizontal (respectively
  vertical) annulus between two adjacent $\alpha$-circles
  (respectively $\beta$-circles) should have two punctures. A toroidal grid diagram
  represents the complement of a link in $S^3$, with meridional
  sutures on the link components. See
  Figure~\ref{fig:toroidal-trefoil}. Toroidal grid diagrams have received
  a lot of attention because, as we will discuss in Section~\ref{sec:grid-diags}, their Heegaard
  Floer invariants have nice combinatorial
  descriptions~\cite{MOS06:CombinatorialDescrip}.
\end{example}

\begin{figure}
  \centering
  \includegraphics{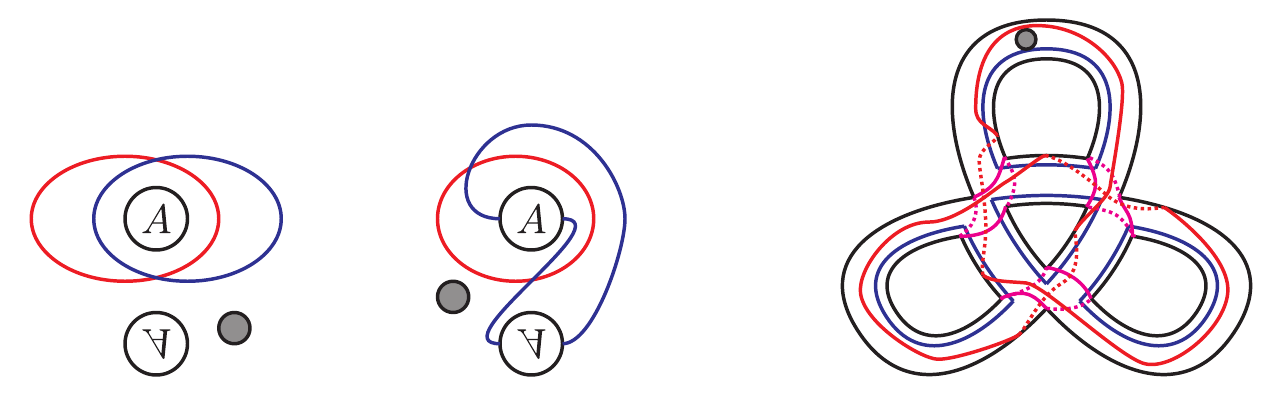}
  \caption{\textbf{Heegaard diagrams for 3-manifolds with $S^2$
      boundary.} Left: a Heegaard diagram for $S^2\times S^1$. Center:
    a Heegaard diagram for $\RR P^3$. Right: a Heegaard diagram for a
    surgery on the trefoil. The $\alpha$ circles are red, $\beta$
    circles are blue, and the labeled empty circles indicate
    handles. (So, the first two pictures lie on punctured tori, and
    the third on a punctured surface of genus $4$.)}
  \label{fig:Heegaard-diags}
\end{figure}

\begin{example}
  Suppose $Y$ is a closed manifold. Fix a \emph{Heegaard splitting}
  for $Y$, i.e., a decomposition $Y=\HB_1\cup_\Sigma \HB_2$, where the
  $\HB_i$ are handlebodies. We can obtain Heegaard diagrams for
  $Y\setminus \bD^3$ as follows. Suppose $\Sigma$ has genus $g$. Fix
  pairwise-disjoint circles $\alpha_1,\dots,\alpha_g\subset \Sigma$ so that:
  \begin{itemize}
  \item Each $\alpha_i$ bounds a disk in $\HB_1$ and
  \item The $\alpha_i$ are linearly independent in $H_1(\Sigma)$.
  \end{itemize}
  Fix circles $\beta_i$ with the same property, but with $\HB_2$ in
  place of $\HB_1$. Let $\Sigma'$ be the result of deleting a disk $D$
  from $\Sigma$ (chosen so that $D$ is disjoint from the $\alpha_i$
  and $\beta_i$). Then
  $(\Sigma',\alpha_1,\dots,\alpha_g,\beta_1,\dots,\beta_g)$ is a
  sutured Heegaard diagram for $Y\setminus\bD^3$ with a single suture
  on the $S^2$ boundary.  See Figure~\ref{fig:Heegaard-diags} for
  some examples.

  In the early days of the subject, these were the only kinds of
  diagrams considered in Heegaard Floer homology.
\end{example}

\begin{figure}
  \centering
  \includegraphics{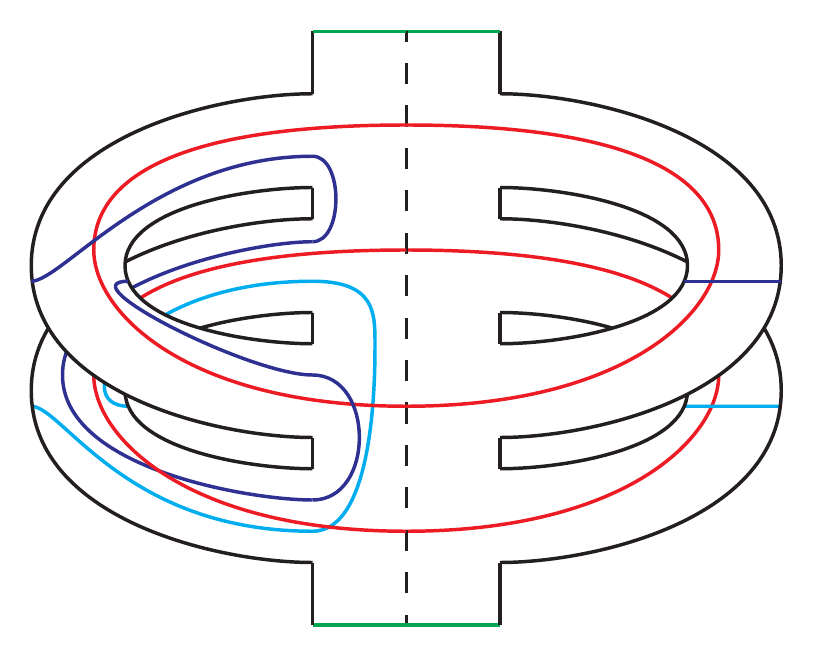}
  \caption{\textbf{Sutured Heegaard diagrams for fibered knot
      complements.} This is a Heegaard diagram for the genus $1$,
    fibered knot with monodromy $ab^{-1}$. The $\alpha$-circles are in
    red and the $\beta$-circles are in blue. The two black arcs in the
    boundary are meant to be glued together in the obvious way.}
  \label{fig:open-book}
\end{figure}

\begin{example}\label{eg:open-book}
  Suppose $K$ is a fibered knot in $Y$, with fiber surface $F$ and
  monodromy $\phi\co F\to F$. Divide $\bdy F$ into two sub-arcs, $A$
  and $B$, so that $\bdy F=A\cup B$ and $A\cap B=\bdy A=\bdy
  B$. Choose $\phi$ so that $\phi(A)=A$ and $\phi(B)=B$.

  Choose $2k$ disjoint, embedded arcs $a_1,\dots,a_{2k}$ in $F$ with
  boundary in $A$, giving a basis for $H_1(F,\bdy F)$. Let
  $b_1,\dots,b_{2k}$ be a set of dual arcs to $a_1,\dots,a_{2k}$, with
  boundary in $B$. (That is, $a_i$ and $b_i$ intersect transversely in
  a single point and $a_i\cap b_j=\emptyset$ if $i\neq j$.)

  Let $\Sigma=[F\cup (-F)]\setminus\nbd(A\cap B)$ be the result of
  gluing together two copies of $F$ and deleting a neighborhood of the
  endpoints of $A$. Let $\alpha_i=a_i\cup a_i$ and let $\beta_i=b_i\cup
  \phi(b_i)$. Then
  $(\Sigma,\alpha_1,\dots,\alpha_{2k},\beta_1,\dots,\beta_{2k})$ is a
  sutured Heegaard diagram for $Y\setminus\nbd(K)$, with two
  meridional sutures along $\bdy\nbd(K)$. See Figure~\ref{fig:open-book}.

  To see this, let $f\co (Y\setminus\nbd(K))\to S^1$ be the
  fibration. Write $S^1=[0,\pi]\cup_{\bdy}[\pi,2\pi]$. We can think of
  $\Sigma$ as 
  \[
  \bigl(f^{-1}(0)\bigr)\cup \bigl(f^{-1}(\pi)\bigr)\cup \bigl([0,\pi]\times A\bigr)\cup \bigl([\pi,2\pi]\times B\bigr).
  \]
  Use the monodromy along $[0,\pi]$ to identify $F=f^{-1}(0)$ and
  $-F=f^{-1}(\pi)$. Then each $\alpha_i$ bounds a disk in
  $f^{-1}([0,\pi])$, and each $\beta_i$ bounds a disk in
  $f^{-1}([\pi,2\pi])$.
\end{example}

Notice that the sutured manifolds specified by a Heegaard diagram are
balanced. (We could have specified unbalanced ones by allowing the
number of $\alpha$ and $\beta$ circles to be different and dropping
our restriction on closed components, but we will not be able to
define invariants of such unbalanced diagrams.
\begin{theorem}\label{thm:HD-exists}
  Any balanced sutured manifold $(Y,\Gamma)$ is represented by a
  sutured Heegaard diagram.
\end{theorem}
\begin{proof}[Proof sketch]
  We will build a Morse function $f\co Y\to\RR$ with certain
  properties and use $f$ to construct the Heegaard
  diagram. Specifically, we want a Morse function $f$ so that:
  \begin{enumerate}
  \item $f\co Y\to [0,3]$.
  \item $f^{-1}(0)=R_-$ and $f^{-1}(3)=R_+$.
  \item\label{item:no-ind-0-3} $f$ has no critical points of index $0$ or $3$.
  \item\label{item:self-indexing} $f$ is \emph{self-indexing}, i.e., for any $p\in\Crit(f)$,
    $f(p)=\ind(p)$.
  \item $f|_{\nbd(\Gamma)\subset\bdy Y}\co \nbd(\Gamma)\cong [0,3]\times\Gamma\to
    [0,3]$ is projection (for some choice of identification $\nbd(\Gamma)\cong [0,3]$).
  \end{enumerate}
  To construct such a Morse function, first define $f$ by hand in a
  neighborhood of $\bdy Y$. Extend $f$ to a Morse function on all of
  $Y$; this is possible since Morse functions are generic. Finally,
  move around / cancel critical points to achieve points
  points~(\ref{item:no-ind-0-3}) and~(\ref{item:self-indexing});
  see~\cite{Milnor65:h-cobordism} for a discussion of how to do that.
  
  Fix also a metric $g$, so that $(\nabla f)|_{\nbd(\Gamma)}$ is
  tangent to $\bdy Y$.

  Now, the Heegaard diagram is given as follows:
  \begin{itemize}
  \item $\Sigma=f^{-1}(3/2)$.
  \item The $\alpha$-circles are the ascending (stable) spheres of the
    index $1$ critical points.
  \item The $\beta$-circles are the descending (unstable) spheres of
    the index $2$ critical points.
  \end{itemize}
  It follows from standard results in Morse theory that the resulting
  Heegaard diagram represents the original sutured manifold;
  see~\cite{Milnor65:h-cobordism} or~\cite{Milnor63:MorseTheory} for the relevant techniques.
\end{proof}

\begin{figure}
  \centering
  \includegraphics{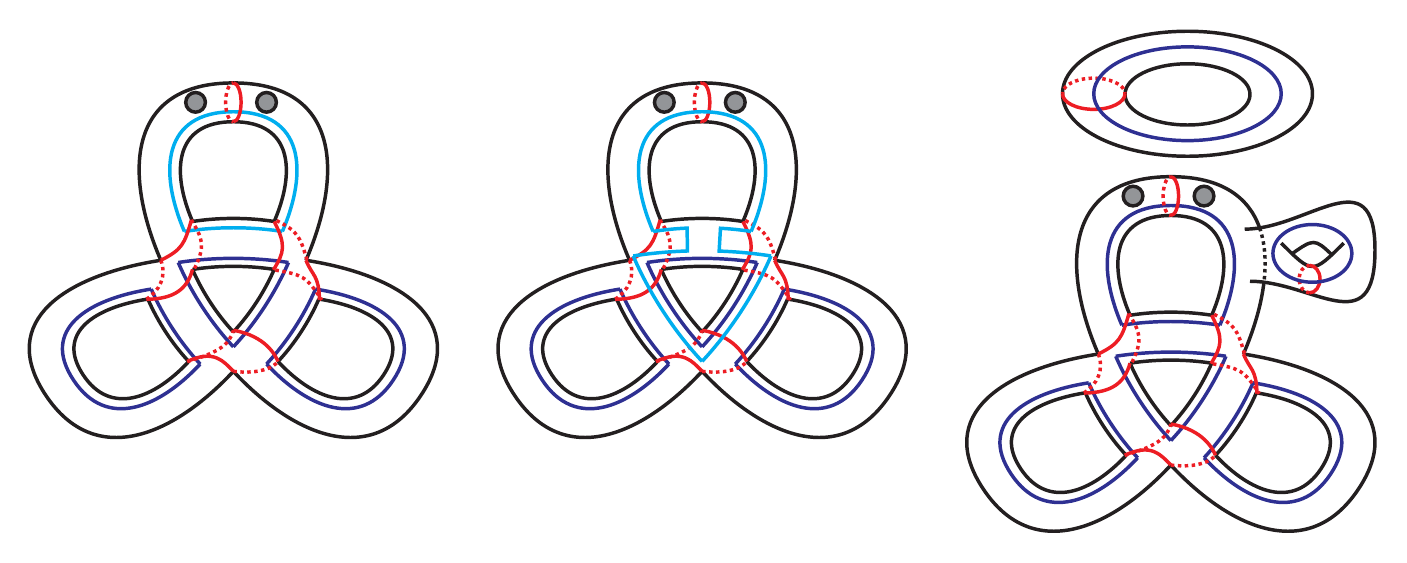}
  \caption{\textbf{Heegaard moves.} Left: a sutured Heegaard diagram
    for the trefoil complement. Center: the result of a
    handleslide among the $\beta$ circles. Right: the standard
    diagram used in the stabilization move, and the result of a
    stabilization.}
  \label{fig:HeegaardMoves}
\end{figure}

We will associate an abelian group $\SFH(\HD)$ to each sutured
Heegaard diagram $\HD$. To prove that these groups depend only on
$Y(\HD)$ (which we will not actually do), it is useful to have a set
of moves connecting any two sutured Heegaard diagrams:
\begin{theorem}\label{thm:heegaard-moves}
  If $\HD$ and $\HD'$ represent homeomorphic sutured manifolds then
  $\HD$ and $\HD'$ can be made homeomorphic by a sequence of the
  following moves:
  \begin{itemize}
  \item Isotopies of $\alphas$ and $\betas$.
  \item Handleslides of one $\alpha$-circle over another or one
    $\beta$-circle over another. (See Figure~\ref{fig:HeegaardMoves}.)
  \item Stabilizations and destabilizations, i.e., taking the
    connected sum with the diagram in Figure~\ref{fig:HeegaardMoves}.
  \end{itemize}
\end{theorem}
Again, the proof I know uses Morse theory.

\begin{remark}
  If one wants to study maps on sutured Floer homology associated to
  cobordisms~\cite{Juhasz:Cobordism}, one needs a more refined
  statement than
  Theorem~\ref{thm:heegaard-moves}. See~\cite{JT:Naturality}.
\end{remark}

\subsection{Holomorphic disks in the symmetric product and \texorpdfstring{$\SFH$}{SFH}}
Brief version:
\begin{definition}(\cite{Juhasz06:Sutured} following~\cite{OS04:HolomorphicDisks}) Fix a sutured Heegaard diagram
  $(\Sigma,\alpha_1,\dots,\alpha_n,\beta_1,\dots,\beta_n)$. Then
  $\SFH(\HD)$ is the Lagrangian intersection Floer homology of
  \[
  T_\alpha=\alpha_1\times\cdots\times\alpha_n,\
  T_\beta=\beta_1\times\cdots\times \beta_n\subset \Sym^n(\Sigma),
  \]
  the $n\th$ symmetric product of $\Sigma$.
\end{definition}
(Recall that the symmetric product $\Sym^n(\Sigma)$ is the quotient of
$\Sigma^{\times n}=\Sigma\times\cdots\times\Sigma$ by the action of
$S_n$ permuting the factors; points in $\Sym^n(\Sigma)$ are unordered
$n$-tuples of points in $\Sigma$, possibly with repetition.)

Longer version:
\subsubsection{Generators}
As its name suggests, the Lagrangian intersection Floer homology is
the homology of a complex $\SFC(\HD)$ generated by the intersection
points between $T_\alpha$ and $T_\beta$:
\[
\SFC(\HD)=\Field\langle T_\alpha\cap T_\beta\rangle.
\]
Unpacking the definition, a point in $T_\alpha\cap T_\beta$ is an
$n$-tuple of points $\{x_i\}_{i=1}^n$, where $x_i\in \alpha_i\cap
\beta_{\sigma(i)}$ for some permutation $\sigma\in S_n$.

\subsubsection{Differential}\label{sec:differential}
The differential, unfortunately, is harder: it counts holomorphic
disks. Recall that an \emph{almost complex structure} on $M$ is a map
$J\co TM\to TM$ so that $J^2=-\Id$. For instance, given a complex
manifold, multiplication by $i$ on the tangent spaces is an almost
complex structure.

To count holomorphic disks, one must work with an appropriate almost complex
structure $J$ on $\Sym^n(\Sigma)$:
\begin{enumerate}
\item The manifold $\Sym^n(\Sigma)$ can be given a reasonably natural
  smooth structure, and in fact has a symplectic form
  $\omega$. (Here, we delete the boundary of $\Sigma$, so $\Sigma$ is a non-compact surface without boundary.) Moreover, the form $\omega$ can be chosen so that
  $T_\alpha$ and $T_\beta$ are Lagrangian~\cite{Perutz07:HamHand}.\footnote{The original
    formulation of Heegaard Floer homology avoided using this fact, by a short
    but clever argument~\cite[Section 3.4]{OS04:HolomorphicDisks}. Also, Perutz's proof is not given in the setting of sutured Floer homology, but rather Ozsv\'ath-Szab\'o's original setting.} In order to know that the moduli spaces of
  holomorphic disks are compact (or have nice compactifications) one
  wants $J$ to be \emph{compatible} with $\omega$, in the sense that
  $\omega(v,Jw)$ is a Riemannian metric.
\item One wants $J$ to be generic enough that the moduli spaces of
  holomorphic disks are transversely cut out.
\end{enumerate}

In practice, one can often work with a \emph{split} almost complex
structure. That is, fix an almost complex structure $j$ on
$\Sigma$. The almost complex structure $j$ induces an almost complex
structure $j^{\times n}$ on $\Sigma^{\times n}$. There is a unique
almost complex structure $\Sym^n(j)$ on $\Sym^n(\Sigma)$ so that the
projection map $\Sigma^{\times n}\to \Sym^n(\Sigma)$ is $(j^{\times
  n},\Sym^n(j))$-holomorphic.

The point of choosing a complex structure is so that we can talk about
holomorphic disks in $\Sym^n(\Sigma)$: a continuous map $u\co \bD^2\to
\Sym^n(\Sigma)$ is \emph{$J$-holomorphic} if $J\circ du = du\circ j$
at all interior points of $\bD^2$, where $j$ is the almost complex
structure on $\bD^2=\{z\in \CC\mid |z|\leq 1\}$ induced by the complex
structure on $\CC$.

\begin{definition}\label{def:disk}
  Given $\x,\y\in T_\alpha\cap T_\beta$, let $\cM(\x,\y)$ be the set
  of non-constant $J$-holomorphic disks $u\co \bD^2\to\Sym^n(\Sigma)$
  so that
  \begin{itemize}
  \item $u(-i)=\x$,
  \item $u(+i)=\y$,
  \item $u(\{z\in\bdy\bD^2\mid \Re(z)\geq 0\})\subset T_\alpha$ and
  \item $u(\{z\in\bdy\bD^2\mid \Re(z)\leq 0\})\subset T_\beta$.
  \end{itemize}
  There is an $\RR$-action on $\cM(\x,\y)$, coming from the
  1-parameter family of conformal transformations of $\bD^2$ fixing
  $\pm i$. (If we identify $\bD^2\setminus\{\pm i\}$ with
  $[0,1]\times\RR$, this $\RR$-action is simply translation in $\RR$.)
\end{definition}

\begin{definition}
  Suppose that $\HD$ represents a RHT sutured 3-manifold. Then define
  $\bdy\co \SFC(Y)\to \SFC(Y)$ by
  \[
  \bdy(\x)=\sum_\y\left(\#\cM(\x,\y)/\RR\right)\y.
  \]
  Here, $\#$ denotes the number of elements modulo $2$; if
  $\cM(\x,\y)/\RR$ is infinite then we declare $\#\cM(\x,\y)/\RR=0$.
\end{definition}

At first glance, this definition looks hard to use: how does one
understand a holomorphic disk in $\Sym^g(\Sigma)$? Somewhat
miraculously, these disks often can be understood, as we will see in
the next section.

If $\HD$ represents a non-RHT sutured 3-manifold, one needs a slightly
more complicated definition. Maps $\bD^2\to \Sym^g(\Sigma)$ decompose
into homotopy classes (corresponding to elements of $H_2(Y)$), and
$\cM(\x,\y)$ is a disjoint union over homotopy classes $\phi$,
$\cM(\x,\y)=\amalg_\phi \cM^\phi(\x,\y)$. One then defines the differential by
$\bdy(\x)=\sum_{\y}\sum_\phi \left(\#\cM^\phi(\x,\y)/\RR\right)\y$,
with the same convention about $\#$ as before. One also needs to add a
requirement on the sutured Heegaard diagram, called
\emph{admissibility}, which ensure that $\#\cM^\phi(\x,\y)=0$ for all
but finitely-many homotopy classes $\phi$. (Admissibility is needed to
get well-defined invariants even if the counts happen to be finite for
other reasons.) 

\subsection{First computations of sutured Floer homology}
\subsubsection{Some \texorpdfstring{$n=1$}{n=1} examples}
If $n=1$ we are just looking at disks in $\Sym^1(\Sigma)=\Sigma$. 

\begin{lemma}\label{lem:moduli-in-Sigma}
  The $0$-dimensional moduli spaces of holomorphic disks in
  $(\Sigma,\alpha\cup\beta)$ correspond to isotopy classes of
  orientation-preserving immersions $\bD\to\Sigma$ with boundary as
  specified in Definition~\ref{def:disk}, $90^\circ$ corners at $x$ and $y$, and which are smooth immersions at all other boundary points.
\end{lemma}
(This follows from the Riemann mapping theorem---exercise.)

Here are some examples.

\begin{figure}
  \centering
  \includegraphics{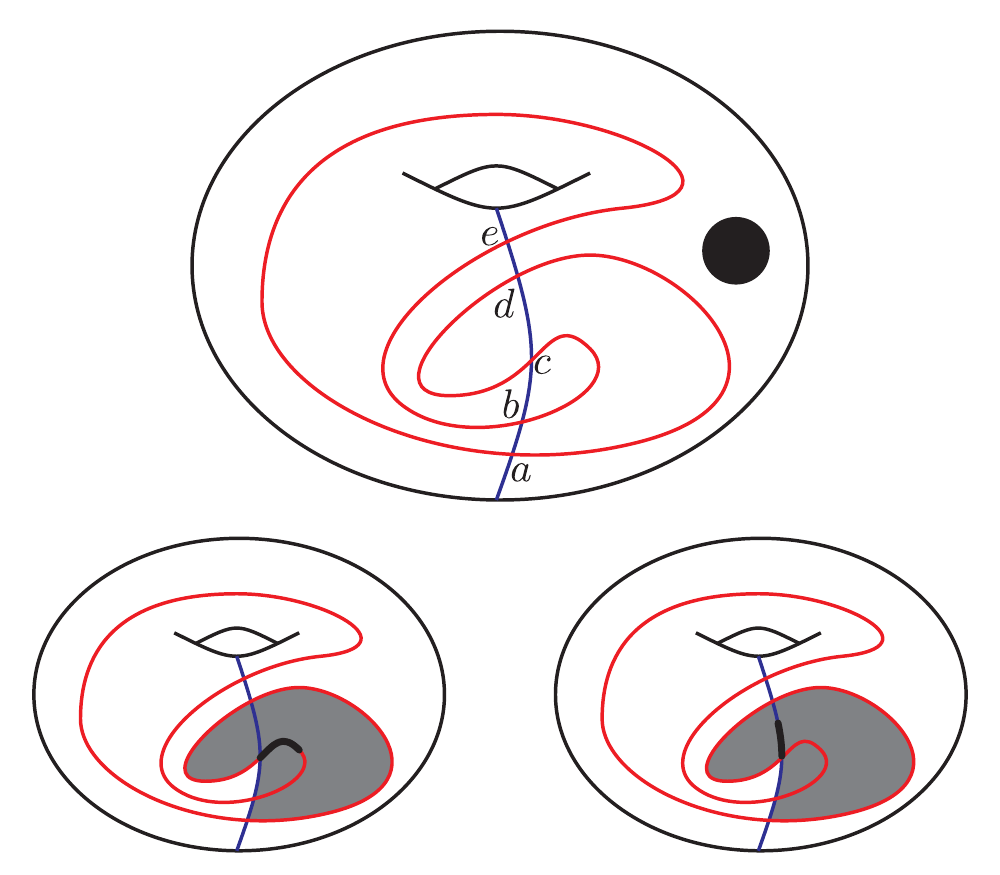}
  \caption{\textbf{A genus-1 Heegaard diagram for
      $S^3\setminus\bD^3$.} Top: the diagram, with generators
    labeled. The big, black disk indicates a hole in $\Sigma$. Bottom:
    a hint of why $\bdy^2=0$. This diagram is adapted
    from~\cite{LOTnotes}.}%
  \label{fig:S3HD-moduli}
\end{figure}

Consider the diagram in Figure~\ref{fig:S3HD-moduli}. This represents
$S^3\setminus \bD^3=\bD^3$, with a single suture on the boundary
$S^2$. The complex $\SFC(\HD)$ has five generators, $a,b,c,d,e$. The
differential is given by 
\begin{align*}
  \bdy(a)&=b+d & \bdy(b)&=c & \bdy(c)&=0\\
  \bdy(d)&=c & \bdy(e)&=b+d
\end{align*}
or graphically
\[
\xymatrix{
  a\ar[d]\ar[drr] & & e\ar[dll]\ar[d]\\
  b\ar[dr] & & d\ar[dl]\\
  & c &
}
\]
So,
\[
\SFH(\bD^3)\cong \Field.
\]

See Figure~\ref{fig:S3HD-moduli} for a hint of why $\bdy^2=0$,
and~\cite[Section 3.1]{LOTnotes} for further discussion of this point.

Note that the fact that the maps must be orientation-preserving means
that the disk from $a$ to $b$ can \emph{not} be read backwards as a
disk from $b$ to $a$.

\begin{figure}
  \centering
  \includegraphics{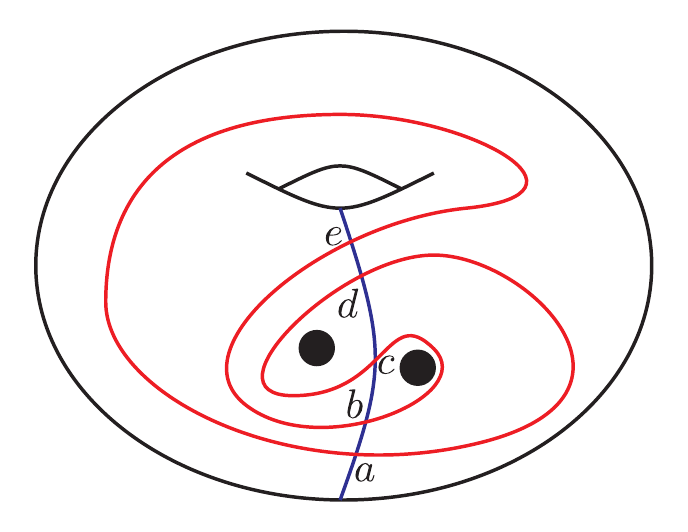}
  \caption{\textbf{A diagram for the figure-8 complement.} The two
    black disks indicate holes in $\Sigma$.}
  \label{fig:fig-8-comp}
\end{figure}
Next, consider the Heegaard diagram in
Figure~\ref{fig:fig-8-comp}. This is the same as
Figure~\ref{fig:S3HD-moduli}, except with different holes. With the
new holes, the differential becomes trivial: the disks we counted
before now have holes in them.  So,
\[
\SFH(S^3\setminus 4_1,\Gamma)\cong (\Field)^5.
\]

These examples can be generalized to compute the Floer homology of
(the complement of) any 2-bridge knot or, more generally, any
(1,1)-knot.

\subsubsection{A stabilized diagram for \texorpdfstring{$\bD^3$}{the three-ball}}\label{sec:genus-2-ball}
Consider the diagram $\HD$ in Figure~\ref{fig:big-D3}. This diagram
again represents $\bD^3$, but now has genus $2$. The complex
$\SFC(\HD)$ has three generators: $\{r,v\}$, $\{s,v\}$ and
$\{t,v\}$. (Notice that one of the $\alpha$-circles is disjoint from
one of the $\beta$-circles, reducing the number of generators.) Since
$\SFH(\HD)=\SFH(\bD^3)=\Field$, the differential must be nontrivial.

There are no obvious bigons in the diagram (or in $\Sym^1(\Sigma)$),
but there is a disk in $\Sym^2(\Sigma)$. Consider the shaded region
$A$ in the middle picture in Figure~\ref{fig:big-D3}. Topologically,
$A$ is an annulus; it inherits a complex structure from the complex
structure on $\Sigma$. I want to produce a holomorphic map $\bD^2\to
\Sym^2(A)$ giving a term $\{s,v\}$ in $\bdy\{t,v\}$. Consider the
result $A_d$ of cutting $A$ along $\alpha_2$ starting at $v$ for a
distance $d$. The key point is the following:

\begin{lemma}\label{lem:annulus-involution}
  There is (algebraically) one length $d$ of cut so that $A_d$ admits a
  holomorphic involution $\tau$ which takes $\alpha$-arcs to $\alpha$-arcs
  (and $\beta$-arcs to $\beta$-arcs and corners to corners).
\end{lemma}
This is an adaptation of the proof of~\cite[Lemma
9.4]{OS04:HolomorphicDisks}. See Exercise~\ref{exercise:annulus}.

Given Lemma~\ref{lem:annulus-involution}, we can construct the map
$u\co \bD^2\to \Sym^2(A)$ as follows. The quotient $A_d/\tau$ is
analytically isomorphic to $\bD^2$, via an isomorphism taking the
image of $t$ and one copy of $v$ to $-i$ and the image of $s$ and the
other copy of $v$ to $+i$ (and hence the $\alpha$-arc to the right
half of $\bdy\bD^2$). This gives a 2-fold branched cover $u_\bD\co
A_d\to \bD^2$. Now, the map $u$ sends a point $x\in \bD^2$ to
$u_\bD^{-1}(x)\in \Sym^2(A)$. It is immediate that $u$ is holomorphic with respect to the split almost complex structure.

\begin{figure}
  \centering
  \includegraphics{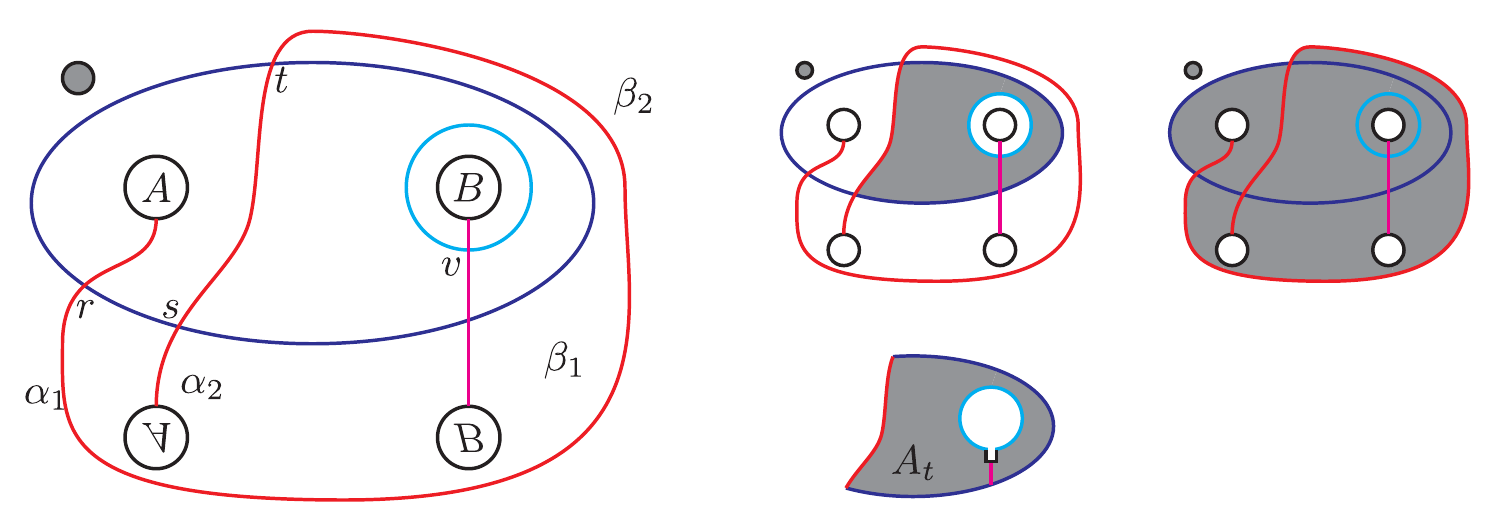}
  \caption{\textbf{A more complicated diagram for $\bD^3$.} The
    $\alpha$-circles are in red and the $\beta$-circles are in
    blue; intersection points which form parts of generators are
    labeled. On the right are two interesting domains, the first from
    $\{t,v\}$ to $\{s,v\}$ and the second from $\{r,v\}$ to
    $\{t,v\}$. The annulus $A_d$ is also shown.}
  \label{fig:big-D3}
\end{figure}

This example illustrates an important principle: any holomorphic disk
$u\co (\bD^2,\bdy\bD^2)\to (\Sym^g(\Sigma),T_\alpha\cup T_\beta)$ has
a shadow in $\Sigma$, in the form of an element of
$H_2(\Sigma,\alphas\cup\betas)$ (i.e., a cellular $2$-chain). This
shadow is called the \emph{domain} of the disk $u$. The multiplicity
of the domain $D(u)$ at a point $p\in \Sigma$ is given by the
intersection number $u\cdot [\{p\}\times \Sym^{g-1}(\Sigma)]$.

Note that the domain has multiplicity $0$ near $\bdy
\Sigma$. Moreover, it follows from \emph{positivity of
  intersections}~\cite{MicallefWhite95:intersection-positivity} that
the coefficients in the domain of a holomorphic $u$ are always
non-negative (at least if one works with an almost complex structure
on $\Sym^g(\Sigma)$ which is close to a split one, or agrees with a
split one on a large enough subset of $\Sym^g(\Sigma)$ so that
$\{z\}\times\Sym^{g-1}(\Sigma)$ is holomorphic for at least one $z$ in
each component of $\Sigma\setminus(\alphas\cup\betas)$; in Heegaard
Floer theory one always makes this restriction). Finally, the domain
has a particular kind of behavior near the generators connected by
$u$: if $u$ connects $\x$ to $\y$ then $\bdy(\bdy D(u)\cap
\alphas)=\y-\x=-\bdy(\bdy D(u)\cap \betas)$.

From these observations, it is fairly easy to see that the only other
possible domain of a holomorphic curve is shown on the far right of
Figure~\ref{fig:big-D3}. This domain connects $\{r,v\}$ to
$\{t,v\}$. But a curve in this homotopy class would violate
$\bdy^2=0$, so the algebraic number of such curves is $0$.

It turns out that one can read the dimension of the moduli space of
disks from the domain $D(u)$: see~\cite[Corollary
4.10]{Lipshitz06:CylindricalHF}.

Of course, in general, computations are more complicated: domains do
not need to be planar (the domain in the right of
Figure~\ref{fig:big-D3} is not planar), and branched covers of degree
greater than $2$ are harder to analyze. Because direct computations
are so hard, there has been a lot of interest in both theoretical and
practical techniques for computing Heegaard Floer homology.

\subsubsection{Grid diagrams}\label{sec:grid-diags}
Consider a toroidal grid diagram $\HD=(\Sigma,\alphas,\betas)$ as in Example~\ref{eg:grid-diags}, and
let $n$ be the number of $\alpha$-circles (which is, of course, also
the number of $\beta$-circles). Since each $\alpha_i$ intersects each
$\beta_j$ in a single point, the generators $\{x_i\in \alpha_i\cap
\beta_{\sigma(i)}\}$ correspond to the permutations $\sigma\in
S_n$. (This correspondence is not quite canonical, since we are using
the indexing of the $\alpha$-circles and $\beta$-circles.)

Next, consider two generators $\x$ and $\y$ such that:
\begin{itemize}
\item $\x\cap \y$ consists of $(n-2)$ points.
\item There is a rectangle $r$ in $\Sigma$ so that the lower-left and
  upper-right corners of $r$ are $\x\setminus \y$, and the upper-left
  and lower-right corners of $r$ are $\y\setminus \x$. (This is a
  meaningful statement.)
\item The interior of $r$ is disjoint from $\x$ (and hence also from $\y$).
\end{itemize}
We will say that $\x$ and $\y$ are \emph{connected by an empty
  rectangle}, and call $r$ an \emph{empty rectangle from $\x$ to $\y$}.

Given an empty rectangle $r$, we can find a holomorphic disk (with
respect to the split complex structure) with domain $r$ as
follows. First, there is a unique holomorphic 2-fold branched cover
$u_\bD\co r\to \bD^2$ sending the $\x$-corners of $r$ to $-i$ and the
$\y$-corners of $r$ to $+i$; see Exercise~\ref{exercise:rect}. (This
map automatically sends the $\alpha$-boundary of $r$ to the right half
of $\bdy \bD^2$ and the $\beta$-boundary to the left half.) Since the
preimage of any point in $\bD^2$ is two points in $r$ (counted with
multiplicity---the branch point is a multiplicity-2 point), we can
view $(u_\bD)^{-1}$ as a map $\bD^2 \to \Sym^2(r)$. There is an
inclusion $\Sym^2(r)\hookrightarrow \Sym^2(\Sigma)\hookrightarrow
\Sym^g(\Sigma)$, where the second inclusion sends $p$ to $p\times
(\x\cap \y)$. (Remember: $p$ is a pair of points in $\Sigma$, and
$\x\cap \y$ is an $(n-2)$-tuple of points in $\Sigma$, so $p\times
(\x\cap\y)$ is an $n$-tuple of points in $\Sigma$. Forgetting the
ordering gives a point in $\Sym^n(\Sigma)$.)

Amazingly, these are the only relevant holomorphic curves in the grid
diagram:
\begin{theorem}\cite{MOS06:CombinatorialDescrip} 
  The rigid holomorphic disks in the symmetric product of a toroidal grid diagram correspond
  exactly to the empty rectangles. In particular, the differential on
  $\SFC(\HD)$ counts empty rectangles in $(\Sigma,\alphas\cup\betas)$.
\end{theorem}
The proof turns out not to be especially hard: it uses an index
formula and some combinatorics to show that the domain of a rigid
holomorphic curve in a toroidal grid diagram must be a
rectangle. The result, however, is both surprising and useful. 

A similar construction is possible for other
3-manifolds~\cite{SarkarWang07:ComputingHFhat}. There is also a
forthcoming textbook about grid diagrams and Floer homology~\cite{OSS:book}.

\subsection{First properties}
\begin{theorem}\cite[Theorem 7.2]{Juhasz06:Sutured}
  The map $\bdy\co \SFC(\HD)\to\SFC(\HD)$ satisfies $\bdy^2=0$.
\end{theorem}
This follows from ``standard techniques''.  The differential $\bdy$ is
defined by counting $0$-dimensional moduli spaces of disks. The
coefficient of $\z$ in $\bdy^2(\x)$ is given by $\#\amalg_\y
\cM(\x,\y)\times\cM(\y,\z)$. One shows that if $\z$ occurs in $\bdy^2(\x)$ then $\cM(\x,\z)$ is the
interior of a compact $1$-manifold with boundary $\amalg_\y
\cM(\x,\y)\times\cM(\y,\z)$; it follows that $\amalg_\y
\cM(\x,\y)\times\cM(\y,\z)$ consists of an even number of points. The proof that $\cM(\x,\z)$ has the desired structure boils down to three parts: 
\begin{enumerate}
\item A transversality statement, that for a generic almost complex
  structure, $\cM(\x,\z)$ is a smooth manifold.
\item A compactness statement, that any sequence of disks in
  $\cM(\x,\z)$ converges either to a holomorphic disk or a broken
  holomorphic disk.
\item A gluing statement, that near any broken holomorphic disk
  one can find an honest holomorphic disk (and, in fact, that near a
  broken disk the space of honest disks is a 1-manifold).
\end{enumerate}

\begin{theorem}\cite[Theorem 7.5]{Juhasz06:Sutured}
  Up to isomorphism, $\SFH(\HD)$ depends only on the (isomorphism
  class of the) sutured $3$-manifold $(Y,\Gamma)$ represented by
  $\HD$.
\end{theorem}
The proof, which is similar to the invariance proof
in~\cite{OS04:HolomorphicDisks}, is broken into three parts:
invariance under isotopies and change of almost complex structure;
invariance under handleslides; and invariance under
stabilization (see Theorem~\ref{thm:heegaard-moves}). Stabilization is easy: it suffices to stabilize near a
boundary component, in which case the two complexes are
isomorphic. Isotopy invariance follows from standard techniques in
Floer theory: one considers moduli spaces of disks with boundary on a
family of moving Lagrangians. Handleslide invariance is a little more
complicated---one uses counts of certain holomorphic triangles (rather
than bigons) to define the relevant maps---but fits nicely with
the modern philosophy of Fukaya categories.

\subsubsection{Decomposition according to \texorpdfstring{$\SpinC$}{spin-c} structures}\label{sec:spinc}
Notice in the example of $S^3\setminus (4_1)$ that there were generators
not connected by any topological disk (immersed or otherwise). This
relates to the notion of $\SpinC$-structures.
\begin{definition}
  Fix a sutured manifold $(Y,\Gamma)$.  Call a vector field $v$ on $Y$
  \emph{well-behaved} if:
  \begin{itemize}
  \item $v$ is non-vanishing.
  \item On $R_+$, $v$ points out of $Y$.
  \item On $R_-$, $v$ points into $Y$.
  \item Along $\gamma$, $v$ is tangent to $\bdy Y$ (and points from
    $R_-$ to $R_+$).
  \end{itemize}
\end{definition}
(The term ``well-behaved'' is not standard.)

\begin{definition}\label{def:homologous-vf}\cite[Definition 4.2]{Juhasz06:Sutured}
  Fix $Y$ connected and a ball $\bD^3$ in the interior of $Y$. We say
  well-behaved vector fields $v$ and $w$ on $Y$ are \emph{homologous}
  if $v|_{Y\setminus\bD^3}$ and $w|_{Y\setminus \bD^3}$ are isotopic
  (through well-behaved vector fields). This is (obviously) an
  equivalence relation. Let $\SpinC(Y,\Gamma)$ denote the set of
  homology classes of well-behaved vector fields; we refer to elements of
  $\SpinC(Y,\Gamma)$ as \emph{$\SpinC$-structures} on $Y$. For $Y$
  disconnected we define
  $\SpinC(Y,\Gamma)=\prod_i\SpinC(Y_i,\Gamma_i)$, where the product is
  over the connected components of $Y$.
\end{definition}

Juh\'asz's Definition~\ref{def:homologous-vf} is inspired by Turaev's
work~\cite{Turaev97:spinc} and the analogous construction
in the closed case from~\cite{OS04:HolomorphicDisks}.

\begin{lemma}\label{spinc-torseur}\cite[Remark 4.3]{Juhasz06:Sutured}
  $\SpinC(Y,\Gamma)$ is a torseur for (affine copy of) $H_1(Y)\cong
  H^2(Y,\bdy Y)$.
\end{lemma}

The first reason $\SpinC$ structures are of interest to us is the following:
\begin{lemma}\cite[Corollary 4.8]{Juhasz06:Sutured}
  There is a map $\spinc\co T_\alpha\cap T_\beta\to \SpinC(Y)$ with
  the property that $\spinc(\x)=\spinc(\y)$ if and only if $\x$ and
  $\y$ can be connected by a bigon (Whitney disk) in
  $(\Sym^g(\Sigma),T_\alpha,T_\beta)$.
\end{lemma}

The map $\spinc$ is not hard to construct from the Morse theory
picture. Start with the gradient vector field $\nabla f$. A generator
$\x$ specifies an $n$-tuple $\{\eta_i\}$ of flow lines connecting the
index $1$ and $2$ critical points. The vector field $\nabla
f|_{Y\setminus \nbd\{\eta_i\}}$ extends to a non-vanishing vector
field on all of $Y$ (easy exercise), which in turn specifies the
$\SpinC$-structure $\spinc(\x)$.

\begin{corollary}
  $\SFH(Y,\Gamma)$ decomposes as a direct sum over $\SpinC$ structures
  on $Y$:
  \[
  \SFH(Y,\Gamma)=\bigoplus_{\spinc\in\SpinC(Y,\Gamma)}\SFH(Y,\Gamma,\spinc).
  \]
\end{corollary}

In fact, $\SFH(Y,\Gamma)$ has a grading by homotopy classes of
well-behaved vector fields. There is a free $\ZZ$-action on
the set of homotopy classes of well-behaved vector fields, so that the
quotient is the set of $\SpinC$-structures. If $Y$ is RHT, this action
is free, which we can abbreviate as:
\[
0\to\ZZ\to \{\text{well-behaved vector
  fields}\}/\text{isotopy}\to\SpinC(Y,\Gamma)\to 0.
\]
The differential on $\SFC(Y,\Gamma)$ changes the ``$\ZZ$-component''
of this grading by $1$ (and leaves the ``$\SpinC(Y,\Gamma)$
component'' unchanged, of
course). See\cite{GrippHuang:plane-fields,GH:bord-gr} for more
details.

\subsubsection{Definition of $\HFa$ and $\HFKa$}\label{sec:HFa}
A few important special cases predated sutured Floer homology, and so
have their own names:
\begin{itemize}
\item For $Y$ a closed $3$-manifold, $\HFa(Y)\coloneqq \SFH(Y\setminus
  \bD^3,\Gamma)$, where $\Gamma$ consists of a single circle on
  $S^2$. This is one of Ozsv\'ath-Szab\'o's original \emph{Heegaard
    Floer homology} groups, from~\cite{OS04:HolomorphicDisks}.
\item For $K$ a nullhomologous knot in a closed manifold $Y$,
  $\HFKa(Y,K)\coloneqq\SFH(Y\setminus\nbd(K),\Gamma)$, where $\Gamma$
  consists of two meridional sutures. The group $\HFKa(Y,K)$ is (one
  variant of)
  the \emph{knot Floer homology group} of $K$, and was introduced by
  Ozsv\'ath-Szab\'o~\cite{OS04:Knots} and
  Rasmussen~\cite{Rasmussen03:Knots}. In the special case $Y=S^3$,
  $\HFKa(Y,K)$ is often denoted simply by $\HFKa(K)$.

  For $Y=S^3$, $\SpinC(Y\setminus \nbd(K),\Gamma)\cong \ZZ$
  canonically. So, $\HFKa(K)$ decomposes:
  \[
  \HFKa(K)=\bigoplus_{j\in\ZZ} \HFKa(K,j).
  \]
  The integer $j$ is called the \emph{Alexander grading}.

  There is also a $\ZZ$-valued homological grading, the \emph{Maslov
    grading}. Further,
  \[
  \sum_{i,j}(-1)^it^j\dim\HFKa_{i}(K,j)=\Delta_K(t),
  \]
  the Alexander polynomial of $K$. (Here, $i$ denotes the Maslov grading.)
\item For $L$ a link in $Y$ each of whose components is
  nullhomologous, $\HFLa(Y,L)\coloneqq\SFH(Y\setminus\nbd(L),\Gamma)$,
  where $\Gamma$ consists of two meridional sutures on each component
  of $\bdy\nbd(L)$. Again, in the special case $Y=S^3$, one often
  writes simply $\HFLa(L)$. The group $\HFLa(Y,L)$ is one variant of
  the
  \emph{link Floer homology} of $L$, and was introduced
  in~\cite{OS05:HFL}.
\end{itemize}

Some other, less well-studied variants also predated sutured Floer
homology. For example, (one variant of) Eftekhary's \emph{longitude
  Floer homology}~\cite{Eftekhary05:LongitudeWhitehead} corresponds to
the sutured Floer homology of a knot complement with two longitudinal
sutures.

\subsubsection{Product sutured manifolds}
If $(Y,\Gamma)$ is a product sutured manifold then we can take
$\Sigma=R_-=R_+$, with $0$ $\alpha$ and $\beta$ circles. In this
rather degenerate case, $\Sym^0(\Sigma)$ is a single point, and
$T_\alpha$ and $T_\beta$ are each a single point as well, giving
$\SFC(Y,\Gamma)=\Field$ with trivial differential. There is also a
unique $\SpinC$ structure on $(Y,\Gamma)$. Thus:
\begin{lemma}\label{lem:HF-prod-sut-mfld}\cite[Proposition 9.4]{Juhasz06:Sutured}
  For $(Y,\Gamma)$ a product sutured manifold,
  $\SFH(Y,\Gamma)=\Field$, supported in the unique $\SpinC$ structure
  on $(Y,\Gamma)$.
\end{lemma}

(If you are uncomfortable with $\Sym^0$, stabilize the diagram
once. The computation remains trivial.)

\subsubsection{Product decompositions, disjoint unions, boundary sums and excess \texorpdfstring{$S^2$}{spherical} boundary components}\label{sec:easy-decomps}
In the next lecture we will discuss how sutured Floer homology behaves under surface decompositions; this behavior is key to its utility. As a simple special case, however, consider a product decomposition $(Y,\Gamma)\stackrel{D}{\rightsquigarrow}(Y',\Gamma')$. One can find a Heegaard diagram $(\Sigma,\alphas,\betas)$ for $(Y,\Gamma)$ with the following properties:
\begin{enumerate}
\item $D\cap\Sigma$ consists of a single arc $\delta$ such that
\item $\delta$ is disjoint from $\alphas$ and $\betas$.
\end{enumerate}
(See~\cite[Lemma 9.13]{Juhasz06:Sutured}.) Cutting $\Sigma$ along
$\delta$ gives a sutured Heegaard diagram $(\Sigma',\alphas,\betas)$
for $(Y',\Gamma')$ (Exercise~\ref{exercise:product-decomp}). With respect to these sutured Heegaard diagrams,
there is an obvious correspondence between generators of
$\SFC(Y,\Gamma)$ and $\SFC(Y',\Gamma')$. Moreover, since the domain of
any holomorphic curve has multiplicity $0$ near $\bdy \Sigma$, it
follows that this identification intertwines the differentials on
$\SFC(\HD)$ and $\SFC(\HD')$. (A little argument, using positivity of
intersections, is needed here.) Thus:
\begin{proposition}\label{prop:product-decomp}\cite[Lemma 9.13]{Juhasz06:Sutured}
  If $(Y,\Gamma)$ and $(Y',\Gamma')$ are related by a product
  decomposition then $\SFH(Y,\Gamma)\cong \SFH(Y',\Gamma')$.
\end{proposition}

In a slightly different direction, suppose we have sutured manifolds
$(Y_1,\Gamma_1)$ and $(Y_2,\Gamma_2)$. The disjoint union $(Y_1\amalg
Y_2,\Gamma_1\amalg\Gamma_2)$ is again a sutured manifold. Moreover,
if $\HD_i$ is a Heegaard diagram for $(Y_i,\Gamma_i)$ then
$\HD_1\amalg \HD_2$ is a Heegaard diagram for $(Y_1\amalg
Y_2,\Gamma_1\amalg \Gamma_2)$. The symmetric product
$\Sym^{g_1+g_2}(\Sigma_1\amalg\Sigma_2)$ decomposes as
$\amalg_{i+j=g_1+g_2}\Sym^{i}(\Sigma_1)\times\Sym^{j}(\Sigma_2)$, but
the Heegaard tori lie in the component $\Sym^{g_1}(\Sigma_1)\times
\Sym^{g_2}(\Sigma_2)$. So (choosing an appropriate almost complex
structure), we get an isomorphism of chain complexes
\[
\SFC(\HD_1\amalg \HD_2)\cong \SFC(\HD_1)\otimes \SFC(\HD_2).
\]
Thus:
\begin{proposition}
  $\SFH(Y_1\amalg Y_2,\Gamma_1\amalg \Gamma_2)\cong
  \SFH(Y_1,\Gamma_1)\otimes \SFH(Y_2,\Gamma_2)$.
\end{proposition}

\begin{figure}
  \centering
  \includegraphics{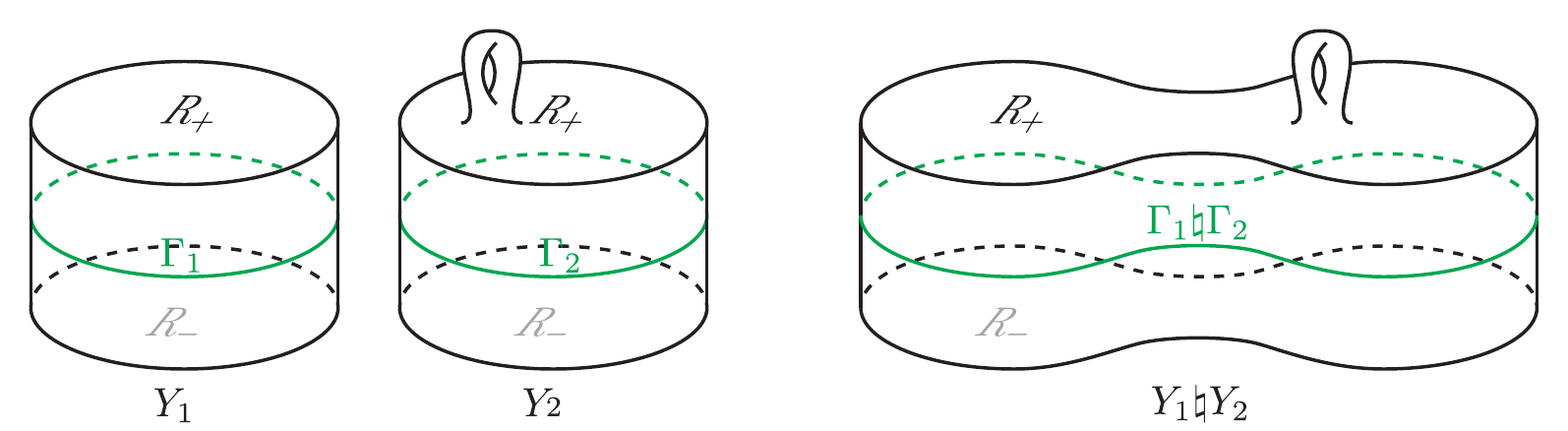}
  \caption{\textbf{The sutured manifold structure on the boundary sum.} The points $q_i$ are marked.}
  \label{fig:sutured-bdy-sum}
\end{figure}

Next, suppose that $(Y_1,\Gamma_1)$ and $(Y_2,\Gamma_2)$ are sutured
manifolds and $\HD_1$ and $\HD_2$ are associated Heegaard
diagrams. Fix a point $p_i\in\bdy\Sigma_i$, corresponding to a point
$q_i\in \Gamma_i\subset \bdy Y_i$. Then we can form the boundary sum
$\HD_1\natural\HD_2$ of $\HD_1$ and $\HD_2$ at the points $p_1$ and
$p_2$. The diagram $\HD_1\natural\HD_2$ represents the boundary sum of
$Y_1$ and $Y_2$, which inherits a sutured manifold structure; see
Figure~\ref{fig:sutured-bdy-sum}. The manifold $Y_1\natural Y_2$
differs from the disjoint union $Y_1\amalg Y_2$ by a product
decomposition, so:
\begin{corollary}
  $\SFH(Y_1\natural Y_2,\Gamma)\cong \SFH(Y_1,\Gamma_1)\otimes
  \SFH(Y_2,\Gamma_2)$.
\end{corollary}
(Of course, this is also easy to prove directly.)

Now, consider the special case of $(Y_1=[0,1]\times S^2,\Gamma_1)$,
where $\Gamma_1$ consists of a single suture on each boundary
component. A Heegaard diagram for $(Y_1,\Gamma_1)$ is shown in
Figure~\ref{fig:Y1}. Here, $\SFC(Y_1,\Gamma_1)$ has two generators,
$x$ and $y$, and there are two disks from $x$ to $y$. (This
computation is easy, since we are in the first symmetric product.)
Consequently, $\bdy(x)=2y=0$, so $\SFH(Y_1,\Gamma_1)=(\Field)^2$.

\begin{figure}
  \centering
  \includegraphics{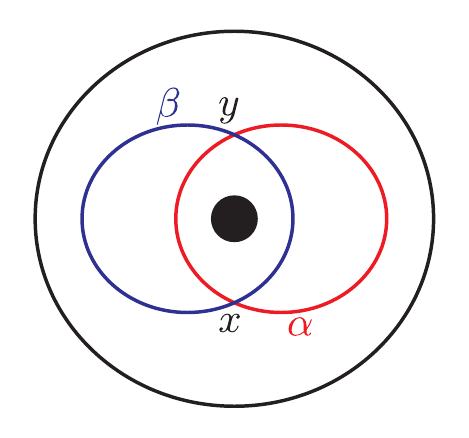}
  \caption{\textbf{A Heegaard diagram for $(Y_1,\Gamma_1)$.} The
    surface $\Sigma$ is an annulus, and there is a single $\alpha$
    circle and a single $\beta$ circle running around the hole.}
  \label{fig:Y1}
\end{figure}

Notice that taking the boundary sum with $(Y_1,\Gamma_1)$ has the effect of introducing a new $S^2$ boundary component, with a single suture. So, we have:
\begin{corollary}\cite[Corollary 9.16]{Juhasz06:Sutured}
  Let $(Y,\Gamma)$ be a sutured manifold and let $(Y',\Gamma')$ be the
  result of deleting a $\bD^3$ from the interior of $Y$ and placing a
  single suture on the resulting boundary component. Then
  $\SFH(Y',\Gamma')\cong \SFH(Y,\Gamma)\otimes(\Field)^2$.
\end{corollary}

\begin{remark}
  At first glance, one might expect $\SFH(Y_1,\Gamma_1)$ to vanish, as
  one can find a Heegaard diagram in which $\alpha$ and $\beta$ are
  disjoint. Note that $(Y_1,\Gamma_1)$ is not RHT, so one is in the
  more complicated situation described at the end of
  Section~\ref{sec:differential}. The need to work with an admissible
  Heegaard diagram is the reason $\SFH(Y_1,\Gamma_1)\neq 0$; but see
  also Exercise~\ref{exercise:vanish}.
\end{remark}

\subsection{Excess meridional sutures}
\begin{proposition}\label{prop:extra-merid}
  If $(Y',\Gamma')$ is obtained from $(Y,\Gamma)$ by replacing a
  suture on a toroidal boundary component with three parallel sutures
  then $\SFH(Y',\Gamma')\cong \SFH(Y,\Gamma)\otimes(\Field)^2$.
\end{proposition}

\begin{corollary}\cite[Proposition 2.5]{MOS06:CombinatorialDescrip}
  If $\HD$ is a grid diagram for a knot $K$ with $n+1$ $\alpha$-circles then 
  \[
  \SFH(\HD)\cong \HFKa(S^3,K)\otimes(\Field)^{2n}.
  \]
\end{corollary}
(For a graded version, $(\Field)^{2n}$ is the exterior algebra on a
bigraded vector space $V$ where one generator is in grading $(0,0)$
and the other is in grading $(-1,-1)$.)

% Probably this follows from~\cite[Lemma 8.9]{Juhasz08:SuturedDecomp}
% or~\cite[Proposition 8.6]{Juhasz08:SuturedDecomp}, though I have not
% thought it through carefully. This is probably a good exercise. See also~\cite{OS05:HFL} and~\cite{MOS06:CombinatorialDescrip}.

\subsection{Suggested exercises}
\begin{enumerate}
\item Convince yourself that Figure~\ref{fig:fig-8-comp} does, in
  fact, represent the complement of the figure-eight knot, with two
  meridional sutures.
\item Convince yourself that Figure~\ref{fig:big-D3} represents
  $\bD^3$ with one suture on the boundary.
\item Generalize Example~\ref{eg:open-book} to the case of fibered
  links. What if we want a diagram for $Y\setminus \bD^3$, where $Y$
  is the 3-manifold in which the knot (or link) $K$ lies, rather than
  $Y\setminus \nbd(K)$?
\item \label{exercise:annulus}Prove Lemma~\ref{lem:annulus-involution}.
\item Show that, given a link $L$ in $S^3$, there is a toroidal grid
  diagram representing $S^3\setminus\nbd(L)$, with some number of
  meridional sutures on each component of $L$. Explicitly find
  toroidal grid diagrams for $(p,q)$ torus knots.
\item Use grid diagrams to compute $\SFH$ for the complement of the
  unknot with $4$ meridional sutures and $6$ meridional sutures, and
  the complement of the Hopf link with $4$ meridional sutures on each
  component.
\item Compute $\SFH$ for the complements of some other $2$-bridge knots.
\item State Lemma~\ref{lem:moduli-in-Sigma} precisely, and prove it.
\item Prove Lemma~\ref{spinc-torseur}.
% \item Verify that we can view $c_1(\spinc)$ as an element of
%   $H^2(Y,\bdy Y)$, and that this element is independent of the choice
%   of vector field $v$ representing $\spinc$.
\item The group $\SpinC(3)$ is isomorphic to $U(2)$. There is a map
  $\SpinC(3)=U(2)\to \SO(3)$ given by dividing out by
  $S^1=\left\{\left(\begin{smallmatrix} e^{i\theta} & 0 \\ 0 &
        e^{i\theta}\end{smallmatrix}\right)\right\}$.

  The usual definition of a $\SpinC$ structure is a principal
  $\SpinC(3)$-bundle $P$ over $Y$, and a bundle map from $P$ to the
  bundle of frames of $Y$, respecting the actions of $\SpinC(3)$ and
  $\SO(3)$ in the obvious sense. (This uses the homomorphism
  $\SpinC(3)\to \SO(3)$ above.) Identify this definition with
  Definition~\ref{def:homologous-vf}.
\item\label{exercise:product-decomp} Suppose that
  $(Y,\Gamma)\stackrel{D}{\rightsquigarrow}(Y',\Gamma')$ is a product
  decomposition and that $\HD=(\Sigma,\alphas,\betas)$ is a Heegaard
  diagram for $(Y,\Gamma)$ and $\delta\subset\Sigma$ is as in
  Section~\ref{sec:easy-decomps}. Let $\Sigma'$ be the result of
  cutting $\Sigma$ along $\delta$. Show that
  $(\Sigma',\alphas,\betas)$ is a sutured Heegaard diagram for $(Y',\Gamma')$.
\item \label{exercise:vanish} Suppose that $(Y,\Gamma)$ is a sutured
  manifold so that one component of $\bdy Y$ is a sphere with $n>1$
  sutures. Prove that $\SFH(Y,\Gamma)=0$. (Warning: to give an honest
  proof, you probably need to know something about admissibility
  conditions.)
\item\label{exercise:lens} Find a genus $1$ Heegaard diagram for the lens space $L(p,q)$
  (or, from the sutured perspective, $L(p,q)\setminus\bD^3$). Use this
  diagram to compute $\HFa(L(p,q))=\SFH(L(p,q)\setminus\bD^3)$.
\item Which surgery on the trefoil is shown in Figure~\ref{fig:Heegaard-diags}?
\item\label{exercise:rect} Let $r$ be a rectangle in the plane, i.e., a topological disk
  with boundary consisting of four smooth arcs. Show that there is a
  unique holomorphic $2$-fold branched cover $r\to \bD^2$ sending the
  corners to $\pm i$. (Hint: start by applying the Riemann mapping
  theorem. Then use the fact that branched double covers $r\to \bD^2$
  correspond to involutions of $r$.)
\end{enumerate}

\section{Surface decompositions and sutured Floer homology}
\renewcommand{\thesec}{Lecture 3} 
\renewcommand{\sectitle}{Surface decompositions and sutured Floer homology}
Recall that to each balanced sutured manifold $(Y,\Gamma)$ we have associated an $\Field$-vector space $\SFH(Y,\Gamma)$. Moreover, $\SFH(Y,\Gamma)$ is a direct sum over (relative) $\SpinC$-structures on $(Y,\Gamma)$,
\[
\SFH(Y,\Gamma)=\bigoplus_{\spinc\in\SpinC(Y,\Gamma)}\SFH(Y,\Gamma,\spinc).
\]

\begin{theorem}\label{thm:Floer-decomp}\cite[Theorem 1.3]{Juhasz08:SuturedDecomp}
  Let $(Y,\Gamma)$ be a balanced sutured manifold and
  $(Y,\Gamma)\stackrel{S}{\rightsquigarrow}(Y',\Gamma')$ a sutured
  manifold decomposition. Suppose that $S$ is good
  (Definition~\ref{def:good}). Then
  \[
  \SFH(Y',\Gamma')\cong \bigoplus_{\spinc\in O(S)}\SFH(Y,\Gamma,\spinc).
  \]
\end{theorem}
(Theorem~\ref{thm:Floer-decomp} holds with ``good'' replaced by
``balanced-admissible''; \cite[Theorem 1.3]{Juhasz08:SuturedDecomp} is
actually the more general statement.)

The notation $O(S)$ needs explanation. A $\SpinC$ structure is called
\emph{outer} with respect to $S$ if it can be represented by a
(non-vanishing) vector field $v$ which is never equal to $-\nu_S$, the
(negative) normal vector field to $S$.  $O(S)$ denotes the set of
outer $\SpinC$ structures. (This definition can be rephrased in terms
of relative Chern classes; see~\cite[Lemma 3.10]{Juhasz08:SuturedDecomp}.)

A key step in proving Theorem~\ref{thm:Floer-decomp} is to study
Heegaard diagrams adapted to the surface decomposition:
\begin{definition}\label{def:adapted-HD}
  Fix a sutured manifold $(Y,\Gamma)$ and a decomposing surface $S$ in
  $Y$. By a \emph{Heegaard diagram for $(Y,\Gamma)$ adapted to $S$} we
  mean a sutured Heegaard diagram $\HD=(\Sigma,\alphas,\betas)$ for
  $(Y,\Gamma)$ together with a subsurface $P\subset \Sigma$ with the
  following properties:
  \begin{enumerate}
  \item The boundary of $P$ is the union $A\cup B$ where $A$ and $B$
    are disjoint unions of smooth arcs.
  \item $\bdy A=\bdy B=A\cap B\subset \bdy\Sigma$.
  \item $A\cap \betas=\emptyset$ and $B\cap \alphas=\emptyset$.
  \item Let $S(P)=P \cup A\times [1/2,1]\cup B\times[0,1/2]\subset
    Y(\HD)=Y$. Then $S(P)$ is isotopic to $S$, where each intermediate
    surface in the isotopy is a decomposing surface.
  \end{enumerate}
  See Figure~\ref{fig:adapted-HD}.
\end{definition}

\begin{figure}
  \centering
  \includegraphics{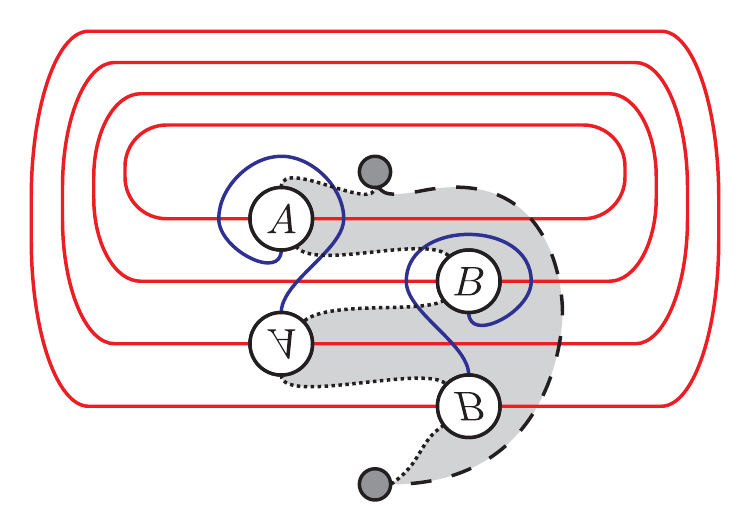}
  \caption{\textbf{A sutured Heegaard diagram adapted to a decomposing
      surface.} This is a Heegaard diagram for the complement of the
    figure 8 knot, and $S(P)$ is a minimal-genus Seifert surface for
    the Figure 8 knot. The polygon $P$ is shaded. $A$ is the
    dashed arc and $B$ is the dotted arc.}
  \label{fig:adapted-HD}
\end{figure}

\begin{proposition}\cite[Proposition 4.4]{Juhasz08:SuturedDecomp}
  Let $(Y,\Gamma)$ be a balanced sutured manifold and $S$ a good
  decomposing surface for $(Y,\Gamma)$. Then there is a sutured
  Heegaard diagram for $(Y,\Gamma)$ adapted to $S$.
\end{proposition}
The proof is similar to but somewhat more intricate than the proof of Theorem~\ref{thm:HD-exists}.

We will return to the proof of Theorem~\ref{thm:Floer-decomp} in
Section~\ref{sec:decomp-proof}.

\subsection{Application: knot genus, (Thurston norm, fiberedness)}
We recall Theorem~\ref{thm:knot-genus}:
\begin{theorem}\cite[Theorem 1.2]{OS04:GenusBounds}
  $\HFKa(S^3,K)$ detects the genus of $K$. Specifically
  \[
  g(K)=\max\{j\mid \HFKa_*(K,j)\neq 0\}.
  \]
\end{theorem}

Similarly:
\begin{theorem}\label{thm:detect-Thurston-norm}
  (\cite[Theorem 2.2]{HeddenNi10:small}, building on \cite[Theorem 1.1]{OS04:GenusBounds})
  For $Y^3$ closed, $\HFa(Y)$ detects the Thurston
  norm: for $h\in H_2(Y)$,
  \[
  x(h)=\max\{\langle c_1(\spinc),h\rangle \mid \HFa(Y,\spinc)\neq 0\}.
  \]
\end{theorem}
Here, $c_1(\spinc)$ denotes the first Chern class of the
$\SpinC$-structure $\spinc$ (which is the same as the Euler class of
the 2-plane field orthogonal to $\spinc$, if we think of $\spinc$ as a
vector field). Ozsv\'ath-Szab\'o proved this result for a
twisted version of Heegaard Floer homology; Hedden-Ni deduce the untwisted statement
using the universal coefficient theorem.

\begin{theorem}\label{thm:fiberedness}
  (\cite[Theorem 1.1]{Ni07:FiberedKnot}, building
  on~\cite[Theorem 1.4]{Ghiggini08:FiberedGenusOne}) $\HFKa(K)$ detects fibered
  knots: $S^3\setminus K$ fibers over $S^1$ if and only if $\sum_i
  \dim\HFKa_i(K,j)=1$.
\end{theorem}
Ni's result is, in fact, more general than
Theorem~\ref{thm:fiberedness}: it holds for nullhomologous knots in
arbitrary 3-manifolds. There is also an analogous statement for closed
3-manifolds~\cite[Theorem 1.1]{Ni09:FiberedMfld}.

In the rest of this section, we will sketch a proof of
Theorem~\ref{thm:knot-genus}. First, the easy direction:
\begin{proposition}\label{prop:easy-direction}\cite{OS04:Knots}
  Fix a Seifert surface $F$ for a knot $K$ in $S^3$.  Then
  $\HFKa_*(K,j)=0$ if $j<-g(F)$ or $j>g(F)$ (where $g(F)$ is the genus
  of $F$).
\end{proposition}
\begin{proof}[Proof sketch]
  We can view $F$ as a good decomposing surface for
  $(S^3\setminus\nbd(K),\Gamma)$, where $\Gamma$ consists of two
  meridional sutures. Choose a Heegaard diagram
  $(\Sigma,\alphas,\betas,P)$ adapted to $F$. It turns out that the
  Alexander grading of a generator $\x\in\HFKa(K)$ is given by
  $|\x\cap P|-g(F)$, where $|\x\cap P|$ denotes the number of points
  in $\x\cap P$ and $g(F)$ is the genus of $F$. (This is, in fact,
  fairly close to the original definition of the Alexander grading
  in~\cite{OS04:Knots}.) It follows that the Alexander grading is
  bounded below by $-g(F)$. For the upper bound we use a symmetry:
  $\HFKa_i(K,j)\cong \HFKa_{i-2j}(K,-j)$~\cite[Proposition
  3.10]{OS04:Knots}.
\end{proof}

\begin{remark}
  In the special case of fibered knots,
  Proposition~\ref{prop:easy-direction} can also be proved using the
  Heegaard diagram from Example~\ref{eg:open-book}. (See
  also~\cite{HKM09:contact}.) That construction can be generalized to
  give a diagram for Proposition~\ref{prop:easy-direction} in
  general. The resulting diagrams are, I think, examples of the ones
  used in this proof (i.e., they are sutured Heegaard diagrams adapted
  to the Seifert surface). One can also prove
  Proposition~\ref{prop:easy-direction} using grid diagrams~\cite{OSS:book}.
\end{remark}

\begin{proof}[Proof of Theorem~\ref{thm:knot-genus}]
  After Proposition~\ref{prop:easy-direction}, it remains to show that
  $\HFKa_*(K,-g(K))\neq 0$. Let $Y_0$ denote the exterior of $K$ and
  let $\Gamma_0$ be two meridional sutures on $\bdy Y$. Fix a
  minimal-genus Seifert surface $F$ for $K$. View $F$ as a decomposing
  surface for $Y_0$ (with $\bdy F$ intersecting each suture once). Let
  $(Y_1,\Gamma_1)$ be the result of a surface decomposition of
  $(Y_0,\Gamma_0)$ along $F$. Since $F$ was minimal genus, the resulting
  sutured manifold is taut. By Theorem~\ref{thm:Floer-decomp},
  $\SFH(Y_1,\Gamma_1)\cong \bigoplus_{\spinc\in
    O(F)}\SFH(Y,\Gamma,\spinc)$. A short argument, similar to the
  argument omitted in the proof of
  Proposition~\ref{prop:easy-direction}, shows that
  $\bigoplus_{\spinc\in O(F)}\SFH(Y,\Gamma,\spinc)=\HFKa_*(K,-g(K))$.

  So, by Theorem~\ref{thm:Floer-decomp}, it suffices to show that
  $\SFH(Y_1,\Gamma_1)$ is nontrivial.  By
  Proposition~\ref{prop:good-hierarchy}, we can find a sequence of
  sutured manifold decompositions
  \[ (Y_1,\Gamma_1)\stackrel{S_1}{\rightsquigarrow}\cdots\stackrel{S_n}{\rightsquigarrow} (Y_n,\Gamma_n)
  \]
  where each $S_i$ is good and $(Y_n,\Gamma_n)$ is a product sutured
  manifold. By Lemma~\ref{lem:HF-prod-sut-mfld},
  $\SFH(Y_n,\Gamma_n)=\Field$. So, applying
  Theorem~\ref{thm:Floer-decomp} $n$ times, $\SFH(Y_1,\Gamma_1)$ has an
  $\Field$ summand.
\end{proof}

Juh\'asz's proof, which we have sketched, of
Theorem~\ref{thm:knot-genus} is quite different from
Ozsv\'ath-Szab\'o's original proof.  Theorem~\ref{thm:fiberedness} can
also be proved using sutured Floer homology, though the argument is
more intricate, and close in spirit to Ni's original
proof. Apparently, at the time of writing there is no known proof of
Theorem~\ref{thm:detect-Thurston-norm} via sutured Floer homology.

\subsection{Sketch of proof of Theorem~\ref{thm:Floer-decomp}}\label{sec:decomp-proof}
We will sketch the proof from~\cite{GrigsbyWehrli10:naturality},
rather than Juh\'asz's original proof
from~\cite{Juhasz08:SuturedDecomp}.  Juh\'asz's original proof, which
uses Sarkar-Wang's \emph{nice
  diagrams}~\cite{SarkarWang07:ComputingHFhat}, is technically
simpler. Grigsby-Wehrli's proof has the advantage that it is more
natural (in a sense they make precise). It is also closer in spirit to
bordered Heegaard Floer theory.

I find it somewhat easier to think about the argument in the ``cylindrical'' formulation of Heegaard Floer homology~\cite{Lipshitz06:CylindricalHF}. This generalizes the description of holomorphic maps $\bD^2\to \Sym^2(\Sigma)$ used in Sections~\ref{sec:genus-2-ball} and~\ref{sec:grid-diags}. Specifically: \begin{proposition}\label{prop:cylindrical} With respect to a split complex structure on $\Sym^g(\Sigma)$, there is a correspondence between holomorphic maps 
  \begin{equation}\label{eq:non-cyl}
  v\co (\bD^2,\bdy\bD^2\cap \{\Re(z)\geq 0\}, \bdy\bD^2\cap \{\Re(z)\leq 0\})\to (\Sym^g(\Sigma),T_\alpha,T_\beta)
  \end{equation}
  and diagrams
  \begin{equation}\label{eq:cylindrical}
    \xymatrix{
      (S,\bdy_a S, \bdy_b S)\ar[r]^{u_\Sigma}\ar[d]^{u_\bD} & (\Sigma,\alphas,\betas)\\
      (\bD^2,\bdy\bD^2\cap \{\Re(z)\geq 0\}, \bdy\bD^2\cap \{\Re(z)\leq 0\})
    }
  \end{equation}
  where $S$ is a Riemann surface with boundary $\bdy S=\bdy_aS\cup \bdy_b S$; $u_\Sigma$ and $u_\bD$ are holomorphic; and $u_\bD$ is a $g$-fold branched cover. 
\end{proposition}
\begin{proof}[Sketch of proof]
  Given a diagram of the form~\eqref{eq:cylindrical} we get a map $\bD^2\to \Sym^g(\Sigma)$ by sending a point $p\in\bD^2$ to $u_\Sigma(u_\bD^{-1}(p))$ (which is $g$ points in $\Sigma$, counted with multiplicity, or equivalently a point in $\Sym^g(\Sigma)$). To go the other way, note that there is a $g$-fold branched cover $\pi\co \Sigma\times \Sym^{g-1}(\Sigma)\to \Sym^g(\Sigma)$ gotten by forgetting the ordering between the $(g-1)$-tuple of points in $\Sigma$ and the one additional point. Given $v\co \bD^2\to \Sym^g(\Sigma)$ as in Formula~\eqref{eq:non-cyl} we can pull back the branched cover $\pi$ to get a branched cover $u_\bD\co S\to \bD^2$. The surface $S$ comes equipped with a map to $\Sigma\times \Sym^{g-1}(\Sigma)$, and projecting to $\Sigma$ gives $u_\Sigma$:
\[
\xymatrix{
& & \Sigma\\
S\ar[r]\ar[urr]^{u_\Sigma}\ar[d]^{u_\bD} & \Sigma\times \Sym^{g-1}(\Sigma)\ar[d]^\pi \ar[ur]_{\pi_\Sigma} & \\
 \bD^2\ar[r]^-v & \Sym^{g}(\Sigma).}
\]
It is fairly straightforward to prove that both constructions give
holomorphic maps (of the specified forms) and that the two
constructions are inverses of each other.
  See~\cite[Section 13]{Lipshitz06:CylindricalHF} for more details (though this idea is not due to me).
\end{proof}

\begin{lemma}\label{lem:decompose-cx}
  Let $(\Sigma,\alphas,\betas,P)$ be a sutured Heegaard diagram
  adapted to a decomposing surface. If $\y$ occurs as a term in $\bdy(\x)$ then $|\x\cap P|=|\y\cap P|$.
\end{lemma}
Lemma~\ref{lem:decompose-cx} follows from various results about $\SpinC$-structures, but it is also fairly easy to prove directly; see Exercise~\ref{exercise:decomp-lemma}.  In fact, a slightly stronger statement holds: if there is a domain connecting $\x$ to $\y$ then $|\x\cap P|=|\y\cap P|$.

\begin{lemma}\label{lem:is-outer}\cite[Lemma 5.4]{Juhasz08:SuturedDecomp}
  With notation as in Lemma~\ref{lem:decompose-cx}, a generator $\x$ represents an outer $\SpinC$-structure if and only if $\x\cap P=\emptyset$.  \end{lemma}

\begin{proof}[Proof of Theorem~\ref{thm:Floer-decomp}]
  Fix a Heegaard diagram $\HD=(\Sigma,\alphas,\betas,P)$ for $(Y,\Gamma)$ adapted to $S$.
  Let $\HD'=(\Sigma',\alphas',\betas')$ be the sutured Heegaard diagram obtained as follows. Topologically,
  \[
  \Sigma'=\bigl(\Sigma\setminus \interior(P)\bigr)\amalg P\amalg P/\sim
  \]
  where $\sim$ identifies the subset $A$ of $\bdy\bigl(\Sigma\setminus \interior(P)\bigr)$ with $A$ in the boundary of the first copy $P_A$ of $P$, and the subset $B$ of $\bdy\bigl(\Sigma\setminus \interior(P)\bigr)$ with $B$ in the boundary of the second copy $P_B$ of $P$. There is a projection map $\pi\co \Sigma'\to \Sigma$, which is $2$-to-$1$ on $P$ and $1$-to-$1$ elsewhere. There is a unique lift $\alphas'$ of the curves $\alphas$ from $\Sigma$ to $\Sigma'$: the lifted curves are disjoint from $P_B$. Similarly, there is a unique lift $\betas'$ of the curves $\betas$, and this lift is disjoint from $P_A$. See Figure~\ref{fig:decomp-HD}

\begin{figure}
  \centering
  \includegraphics{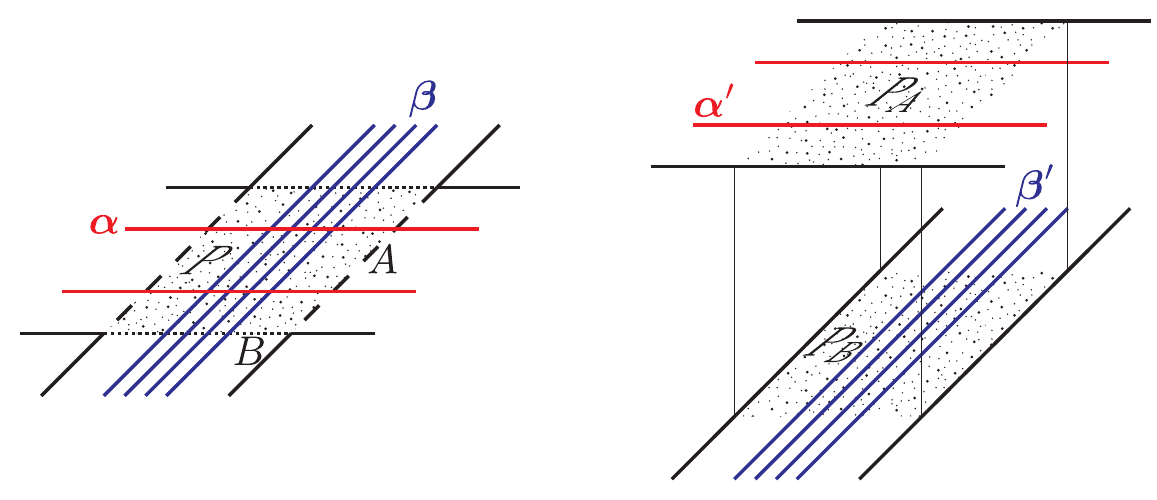}
  \caption{\textbf{A surface decomposition of a Heegaard diagram.} A piece of the original diagram is on the left and the corresponding piece of the decomposed diagram is on the right. The region $P$ and its two preimages $P_A$ and $P_B$ are speckled. The curve $A$ is dashed and $B$ is dotted. A more artistic version of this figure is~\cite[Figure 4]{Juhasz08:SuturedDecomp}.}
  \label{fig:decomp-HD}
\end{figure}

  Let 
  \[
  \SFC_P(\HD)=\langle \{\x\mid \x\cap P=\emptyset\}\rangle \subset\SFC(\HD).
  \]
  By Lemma~\ref{lem:decompose-cx}, $\SFC_P(\HD)$ is a direct summand of the chain complex $\SFC(\HD)$. 

  By Lemma~\ref{lem:is-outer}, an equivalent formulation of Theorem~\ref{thm:Floer-decomp} is:
  
  \emph{There is an isomorphism $\SFH_P(\HD)\cong \SFH(\HD').$}

  This statement has two advantages: it is more concrete, so we can prove it, and it does not make reference to $\SpinC$ structures, which we have not discussed much. It has the disadvantage that it is not intrinsic---it talks about diagrams, not sutured manifolds.

  Notice that $\alphas'\cap \betas'$ corresponds (via the projection $\pi$) to $\alphas\cap\betas\cap (\Sigma\setminus P)$. This induces an identification of generators between $\SFC_P(\HD)$ and $\SFC(\HD')$. We will show that for an appropriate choice of complex structure, this identification intertwines the differentials. (As usual, we are suppressing transversality issues and assuming we can work with split almost complex structures.)

Working in the cylindrical formulation from Proposition~\ref{prop:cylindrical}, suppose $\x$ and $\y$ are generators of $\SFC_P(\HD)$ and that $\y$ occurs in $\bdy(\x)$. Then there is a diagram $\bD\stackrel{u_\bD}{\longleftarrow} S\stackrel{u_\Sigma}{\longrightarrow}\Sigma$ as in Formula~\eqref{eq:cylindrical}. We want to produce a similar diagram, but in $\HD'$.

  The idea is to insert long necks in $\Sigma$ along $A$ and $B$, or equivalently, to pinch $A$ and $B$, decomposing  $\Sigma$  into two parts: $P/\bdy P$ and $\Sigma/P$. (The argument is similar to the first part of the argument in~\cite[Chapter 9]{LOT1}.) Consider a sequence of curves $u_i=(u_{\bD,i},u_{\Sigma,i})$ as above, with respect to a sequence of neck lengths converging to $\infty$.

  \emph{Claim 1.} As $A$ and $B$ collapse, one can find a subsequence of the $u_i$ so that:
  \begin{itemize}
  \item The surfaces $S_i$ converge to a nodal Riemann surface $S_\infty$.
  \item $S_\infty$ has two components, $S_\infty^P$ and $S_\infty^\Sigma$, attached at a collection of boundary points (nodes).
  \item The maps $u_i$ converge to holomorphic maps
    \begin{align*}
    u_{\bD,\infty}^\Sigma&\co S_\infty^\Sigma\to \bD^2 & 
    u_{\bD,\infty}^P&\co S_\infty^P\to \bD^2 \\ 
    u_{\Sigma,\infty}^\Sigma&\co S_\infty^\Sigma\to \Sigma/P & 
    u_{\Sigma,\infty}^P&\co S_\infty^P\to P/\bdy P.
    \end{align*}
  \item The maps $u_{\bD,\infty}^P$ and $u_{\bD,\infty}^\Sigma$ send each side of each node to the same point in $\bdy \bD^2$; that is, $u_{\bD,\infty}$ extends continuously over the nodes.
  \item At each node, $u_{\Sigma,\infty}^P$ and $u_{\Sigma,\infty}^\Sigma$ map to an arc between two $\alpha$- or $\beta$-circles and, further, both sides of the node map to the same such arc.
  \end{itemize}
  See Figure~\ref{fig:degen-curve} for a schematic example.

  \begin{figure}
    \centering
    \includegraphics{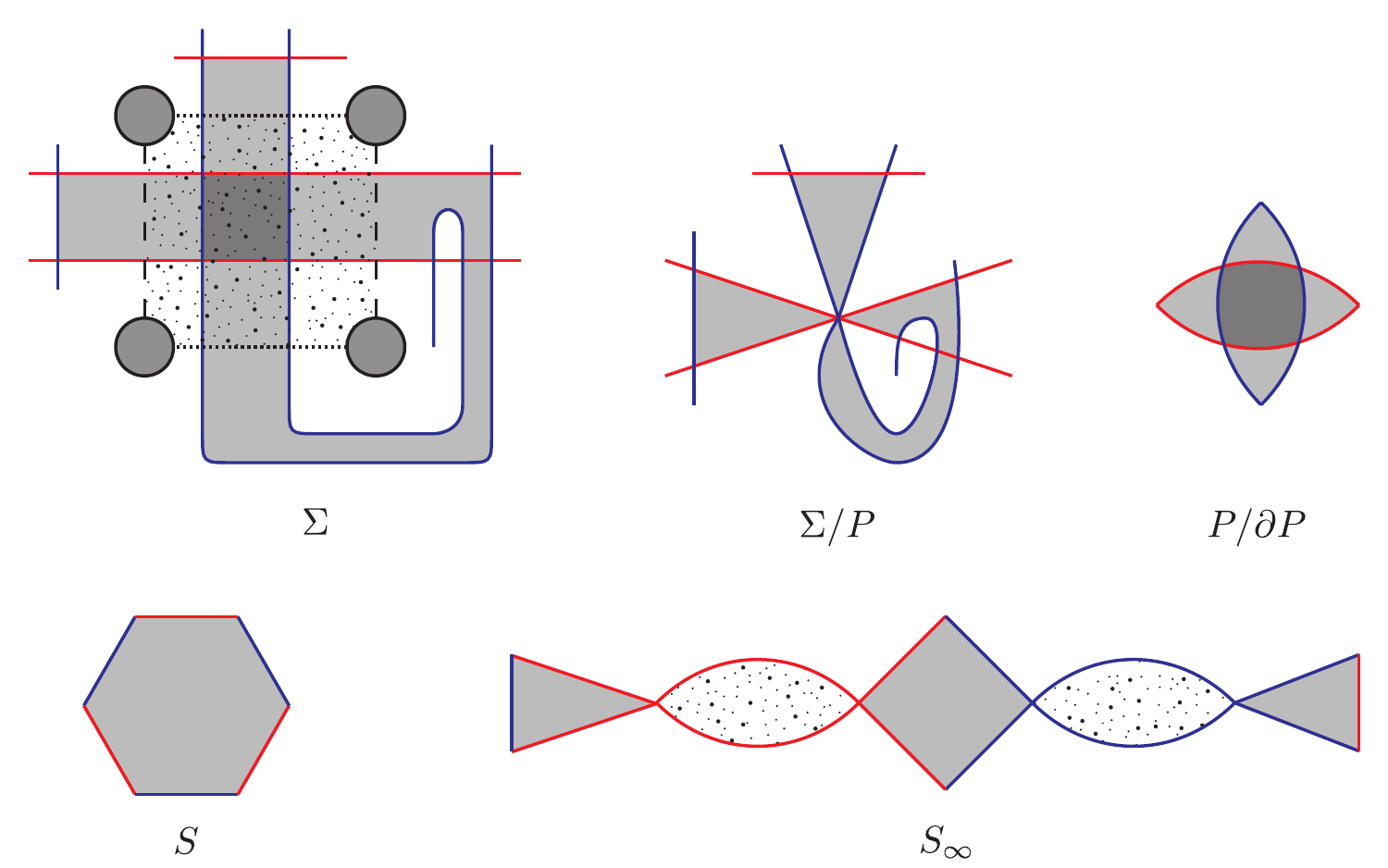}
    \caption{\textbf{A holomorphic curve after degeneration.} The domain of $u$ is shaded, and $P$ is speckled. In the domain of $u$, the darkly-shaded part is covered twice. In $P/\bdy P$, the four corners are identified. The conformal structure of $S$ is not (usually) the one indicated. In $S_\infty$, $S_\infty^\Sigma$ is shaded and $S_\infty^P$ is speckled.}
    \label{fig:degen-curve}
  \end{figure}

  Claim 1 is a version of Gromov's compactness theorem~\cite{Gromov85} (see also~\cite{BEHWZ03:CompactnessInSFT}), though the fact that we are considering maps between surfaces make it considerably easier than the general case.

  \emph{Claim 2.} Near any limiting surface as in Claim 1 there is a sequence of holomorphic curves converging to it.

  Claim 2 is called a gluing theorem. (Again, the fact that we are looking at maps between surfaces means this is a reasonably simple case.)

  Together, Claims 1 and 2 mean that we can use this degenerated surface to compute the differential on $\SFC_P$.

  \emph{Claim 3.} The surface $S_\infty^P$ consists of a disjoint union of bigons (disks with two boundary nodes). The map $u_{\Sigma,\infty}^P$ sends each bigon to a strip in $P$, with boundary either two $\alpha$-circles or two $\beta$-circles. The map $u_{\bD,\infty}^P$ is constant on each bigon.

  Notice that Claim 3 implies that $u_{\Sigma,\infty}^P$ and $u_{\bD,\infty}^P$ can be reconstructed from $u_{\Sigma,\infty}^\Sigma$ and $u_{\bD,\infty}^\Sigma$, and that $u_{\Sigma,\infty}^P$ and $u_{\bD,\infty}^P$ exist if and only if $u_{\Sigma,\infty}^\Sigma$ and $u_{\bD,\infty}^\Sigma$ satisfy certain easy-to-state properties (Exercise~\ref{ex:claim-3}).

  Similar results hold for holomorphic curves in $\Sigma'$, after collapsing the arcs $A$ and $B$ there. The difference is that we now have three components: $\Sigma/P$, $P_A/A$ and $P_B/B$. The analogue of Claim 3 says:

  \emph{Claim 3$'$.} Each of the surfaces $S_\infty^{P_A}$ and $S_\infty^{P_B}$ consists of a disjoint union of bigons. The map $u_{\Sigma,\infty}^{P_A}$ sends each bigon to a strip in $P_A$, with boundary on two $\alpha$-circles. The map $u_{\Sigma,\infty}^{P_B}$ sends each bigon to a strip in $P_B$, with boundary on two $\beta$-circles. The map $u_{\bD,\infty}^P$ is constant on each bigon.

  Again, Claim 3$'$ implies that the curves $u_{\Sigma,\infty}^{P_A}$, $u_{\bD,\infty}^{P_A}$, $u_{\Sigma,\infty}^{P_B}$ and $u_{\bD,\infty}^{P_B}$ can be reconstructed from $u_{\Sigma,\infty}^\Sigma$ and $u_{\bD,\infty}^\Sigma$. In particular, 
   there is an identification between the curves in Claim 3 and the curves in Claim 3$'$. Since we can use these degenerated curves to compute the differentials on $\SFC_P(\HD)$ and $\SFC(\HD')$, this completes the proof.
\end{proof}

% \subsection{Remarks on outer \texorpdfstring{$\SpinC$}{spin-c} structures}

\subsection{Some open questions}
Here are some questions about sutured Floer homology which I think are
open, and which I would find interesting to have answered. (Whether or
not anyone else would find them interesting I cannot say.)
\begin{enumerate}
\item Can one give a proof of Theorem~\ref{thm:detect-Thurston-norm}
  using sutured Floer homology? In this context~\cite[Section
  7.8]{KronheimerMrowka10:sutured} seems
  relevant.
\item Further explore the constructions in~\cite{AE:SFH-minus}, or other ``minus'' variants of sutured Floer homology.
\item What can be said about the next-to-outer $\SpinC$ structures? What topological information do they contain? (Perhaps the pairing theorem in~\cite{Zarev09:BorSut} is relevant, as might be~\cite{LOT:faith}.)
% \item In some sense, the pairing theorem in bordered Heegaard Floer homology (see~\cite{LOT1}) is a generalization of Theorem~\ref{thm:Floer-decomp}. (More precisely, the pairing theorem for bordered-sutured Floer homology~\cite{Zarev09:BorSut} is a strict generalization of Theorem~\ref{thm:Floer-decomp}. Theorem~\ref{thm:Floer-decomp} corresponds to the extremal strands grading.)  Is there a different version of surface decompositions which allows one to state
\end{enumerate}

\subsection{Suggested exercises}
\begin{enumerate}
\item Let $(\Sigma,\alphas,\betas,P)$ be a sutured Heegaard diagram
  adapted to a decomposing surface $S$. How does one compute the genus
  of $S=S(P)$?
\item Deduce Proposition~\ref{prop:product-decomp} from
  Theorem~\ref{thm:Floer-decomp}.
\item Convince yourself that Figure~\ref{fig:adapted-HD} represents
  the figure 8 knot, and that $S(P)$ is a minimal genus Seifert
  surface. (Hint: the figure 8 knot is fibered with monodromy
  $ab^{-1}$.) Give a sutured Heegaard diagram for the trefoil
  complement, adapted to a minimal genus Seifert surface.
\item What is the relationship between the proof of the ``easy
  direction'' of Theorem~\ref{thm:knot-genus}, i.e.,
  Proposition~\ref{prop:easy-direction}, that we gave and the proof
  in~\cite{OS04:Knots}? (That is, are the diagrams used
  in~\cite{OS04:Knots} examples of diagrams adapted to the Seifert
  surface $F$ in the sense of Definition~\ref{def:adapted-HD}?)
\item Prove that if $K$ is a fibered knot then $\HFK_*(K,-g(K))\cong
  \Field$ (i.e., the easy direction of Theorem~\ref{thm:fiberedness}).
\item Fill in the details in the proof of Proposition~\ref{prop:cylindrical}.
\item\label{exercise:decomp-lemma} Prove Lemma~\ref{lem:decompose-cx}.
\item\label{ex:claim-3} In the proof of Theorem~\ref{thm:Floer-decomp}, say precisely what properties $S_\infty^\Sigma$, $u_{\Sigma,\infty}^\Sigma$ and $u_{\bD,\infty}^\Sigma$ must satisfy for the surface $S_\infty^P$ and the maps $u_{\Sigma,\infty}^P$ and $u_{\bD,\infty}^P$ to exist. (See the discussion immediately after Claim 3.)
\end{enumerate}

\section{Miscellaneous further remarks}
\renewcommand{\thesec}{Lecture 4} 
\renewcommand{\sectitle}{Further remarks}
The main goal of this lecture is to draw some connections with the
lecture series on Khovanov homology. As a side benefit, I will mention
another nice applications of Heegaard Floer homology (mostly without
proof, unfortunately).

For most of this talk we will focus on the invariant
$\HFa(Y)=\SFH(Y\setminus \bD^3,\Gamma)$ associated to a closed $3$-manifold
$Y$, as in Section~\ref{sec:HFa}.

\subsection{Surgery exact triangle}
A \emph{framed knot} in a $3$-manifold $Y$ is a knot $K\subset Y$
together with a slope $n$ (isotopy class of essential simple closed
curves) on $\bdy \nbd(K)$. Given a framed knot $(K,n)$ we can do
\emph{surgery} on $(K,n)$ by gluing a thickened disk ($3$-dimensional
$2$-handle) to $Y\setminus\nbd(K)$ along $n$, and then capping the
resulting $S^2$ boundary component with a $\bD^3$. Let $Y_n(K)$ denote
the result of doing surgery to $Y$ along $(K,n)$.

Call a triple of slopes $(n,n',n'')$ in an oriented $T^2$ a
\emph{triad} if it is possible to orient $n$, $n'$, and $n''$ so that
their intersection numbers satisfy $n\cdot n'=n'\cdot n''=n''\cdot
n=-1$ (compare~\cite[Section 2]{BrDCov}).
% , up to isotopy, $n$ and $n'$ intersect in one point
% and $n''$ is obtained from $n\cup n'$ by smoothing the intersection
% point by going along $n'$, turning $90^\circ$ in the positive
% direction, and then going along $n$. See Figure~\ref{fig:triad}. (I
% do not think this is a standard term.) If $(n,n',n'')$ is a triad then
% so are $(n',n'',n)$ and $(n'',n,n')$, but $(n,n'',n')$ is not a triad.
%
% \begin{figure}
%   \centering
%   \includegraphics{triad}  
%   \caption{\textbf{A triad of slopes.} As shown, the three slopes $(\infty,0,1)$ form a triad.}
%   \label{fig:triad}
% \end{figure}
%
Given a knot $K\subset Y$, orient $\bdy\nbd(K)=\bdy (Y\setminus
\nbd(K))$ as the boundary of $(Y\setminus \nbd(K))$.

\begin{theorem}\label{thm:surj-tri}\cite{OS04:HolDiskProperties}
  Let $(n,n',n'')$ be a triad of slopes in $\bdy(Y\setminus \nbd(K))$.
  Then there is an exact triangle
  \[
  \xymatrix{
     & \HFa(Y_n(K)) \ar[dr]& \\
     \ar[ur] \HFa(Y_{n''}(K)) & & \HFa(Y_{n'}(K)).\ar[ll] 
  }
  \]
\end{theorem}

In fact, the same result holds for $\HFa(Y)$ replaced by
$\SFH(Y,\Gamma)$ for any sutured $3$-manifold $(Y,\Gamma)$. The
original proof (from~\cite{OS04:HolDiskProperties}) extends to this
case. This also follows immediately from the exact triangle for
bordered solid tori~\cite[Section 11.2]{LOT1} together with Zarev's
bordered-sutured theory---particularly~\cite[Theorem
3.10]{Zarev:JoinGlue}.

There is always a distinguished slope for $K$, the \emph{meridian},
which bounds a disk in $\nbd(K)$. If $K\subset S^3$ then $K$ also has
a well-defined \emph{longitude}, a slope which is nullhomologous in
$S^3\setminus\nbd(K)$. For knots in $S^3$, therefore, we can identify
slopes with rational numbers, by declaring that $p/q$ corresponds to
$p$ times the meridian plus $q$ times the longitude.

\subsection{Lens space surgery}
Theorem~\ref{thm:surj-tri} has many applications. It is a central tool
in computations of 3-manifold invariants; see for
instance~\cite{JabukaMark08:surf-times-circ} for an intricate
example.

In a different direction, let us consider $3$-manifolds $Y$ for which $\HFa(Y)$ is trivial. First, we must decide what we mean by ``trivial''. To start, we have:
\begin{theorem}\label{thm:chi}
  \cite[Proposition 5.1]{OS04:HolDiskProperties} Given a torsion
  $\SpinC$ structure $\spinc$ on $Y$, $\chi(\HFa(Y,\spinc))=\pm 1$. In
  particular, for $Y$ a rational homology sphere (meaning $H_1(Y;\QQ)=0$), there is a choice of
  absolute $\ZZ/2\ZZ$ grading so that $\chi(\HFa(Y))=|H_1(Y)|$, the
  number of elements of $H_1(Y)$.
\end{theorem}
Suppose $Y$ is a rational homology sphere. Saying $\HFa(Y)$ is trivial, then, should mean $\dim(\HFa(Y))=\chi(\HFa(Y))=|H_1(Y)|$. In this case, we say that $Y$ is an \emph{$L$-space}. The terminology comes from the fact that the lens spaces $L(p,q)$ are all $L$-spaces: $|H_1(L(p,q))|=p$ and $\HFa(L(p,q))\cong (\Field)^p$ (Exercise~\ref{exercise:lens}).

As a first application of Theorem~\ref{thm:surj-tri}, we have:
\begin{corollary}\label{cor:larger-surg}
  Let $S^3_n(K)$ denote $n$-surgery on the knot $K$. 
  If $S^3_n(K)$ is an $L$-space then so is $S^3_m(K)$ for any $m>n$.  
\end{corollary}
\begin{proof}
  By induction, it suffices to prove that $S^3_{n+1}(K)$ is an
  $L$-space.  Applying Theorem~\ref{thm:surj-tri} to the slopes $n$,
  $n+1$ and $\infty$, we see that $\dim\HFa(S^3_{n+1}(K))\leq
  n+1$. Theorem~\ref{thm:chi} gives the opposite inequality, proving
  the result.
\end{proof}

$L$-spaces are fairly rare, though many examples are known. (For example, the branched double cover of any alternating link is an $L$-space~\cite{BrDCov}; this follows from the techniques in Section~\ref{sec:dcov}.)

Via Theorem~\ref{thm:surj-tri} and its refinements (like the surgery formulas from~\cite{OS08:IntSurg,OS11:RatSurg}), one can give restrictions on which surgeries can yield $L$ spaces and, in particular, lens spaces. Perhaps the most dramatic example (so far) is a theorem of Kronheimer-Mrowka-Ozsv\'ath-Szab\'o, originally proved using monopole (Seiberg-Witten) Floer homology:

\begin{theorem}\label{thm:KMOS}
  \cite{KMOS} Suppose that for some $p/q\in\QQ$, $S^3_{p/q}(K)$ is
  orientation-preserving diffeomorphic to the lens space
  $L(p,q)$. Then $K$ is the unknot.
\end{theorem}
Many cases of Theorem~\ref{thm:KMOS} were already known; see the introduction
to~\cite{KMOS} for a discussion of the history.

Note that there are nontrivial knots $K$ in $S^3$ admitting lens space surgeries. For instance, $S^3_{pq+1}(T_{p,q})=L(pq+1,q^2)$. A number of other knots (called \emph{Berge knots}) are known to have lens space surgeries, and many others have $L$-space surgeries (see, for instance,~\cite{HLV:Berge-Gabai}, and its references). So, the following is \emph{false}: if $\HFa(S^3_{p/q}(K))\cong \HFa(S^3_{p/q}(U))$ then $K=U$. In particular, the proof of Theorem~\ref{thm:KMOS} needs (at least) one more ingredient.

In Theorem~\ref{thm:KMOS}, the case of $0$-surgeries is called the \emph{property $R$ conjecture}, and was proved by Gabai~\cite{Gabai87:foliations-2-3}. So, to prove Theorem~\ref{thm:KMOS} for $q=1$, say, it suffices to show that if $S^3_p(K)=L(p,1)$ then $S^3_{p-1}(K)=L(p-1,1)$. This is the opposite direction of induction from Corollary~\ref{cor:larger-surg}.

To accomplish this downward induction, one can either use the \emph{absolute $\QQ$-grading} on $\HFa(Y)$, as in the original proof of Theorem~\ref{thm:KMOS} or, for a quicker proof, the surgery formula from~\cite{OS08:IntSurg}. (In fact, the latter proof is sufficiently simple that Theorem~\ref{thm:KMOS} makes a good exercise when learning the surgery formula.)

Theorem~\ref{thm:KMOS} is one of everyone's favorite applications of low-dimensional Floer theories, and so gets mentioned a lot. For some other striking applications to lens space surgeries, which are fairly accessible from the discussion in these lectures, see for instance~\cite{OS05:surgeries}.

\subsection{The spectral sequence for the branched double cover}\label{sec:dcov}
As another application of the surgery exact triangle (and related
techniques), we discuss a relationship between Heegaard Floer
homology and Khovanov homology:
\begin{theorem}\label{thm:s-seq}
  \cite[Theorem 1.1]{BrDCov} For any link $L$ in $S^3$ there is a spectral
  sequence $\rKh(m(L))\Rightarrow \HFa(\Sigma(L))$.
\end{theorem}
Here, $\rKh(m(L))$ denotes the reduced Khovanov homology of the mirror
of $L$. The manifold $\Sigma(L)$ is the double cover of $S^3$ branched
along $L$. That is, the meridians of $L$ define a canonical
isomorphism $H_1(S^3\setminus\nbd(L))\cong \ZZ^{|L|}$. The composition
\[
\pi_1(S^3\setminus\nbd(L))\to H_1(S^3\setminus\nbd(L))=\ZZ^{|L|}\to\ZZ/2\ZZ,
\]
where the last map sends each basis vector (i.e., meridian) to $1$,
defines a connected double cover $\tilde{Y}$ of $S^3\setminus\nbd(L)$. The
boundary of $\tilde{Y}$ is a union of tori. Each of these tori has a
distinguished meridian---the total preimage of a meridian of the
corresponding component of $L$. Filling in these meridians with
thickened disks and filling the resulting $S^2$ boundary components with
$\bD^3$'s gives the double cover of $S^3$ branched along $L$.

Theorem~\ref{thm:s-seq} has received a lot of
attention; see~\cite[Section 1.2]{LOT:DCov1} for references to related
work. In particular, Kronheimer-Mrowka later used similar ideas to prove that Khovanov homology detects the unknot~\cite{KronheimerMrowka11:detect}.

\begin{proof}[Sketch of Proof of Theorem~\ref{thm:s-seq}]
  \begin{figure}
    \centering
    \includegraphics{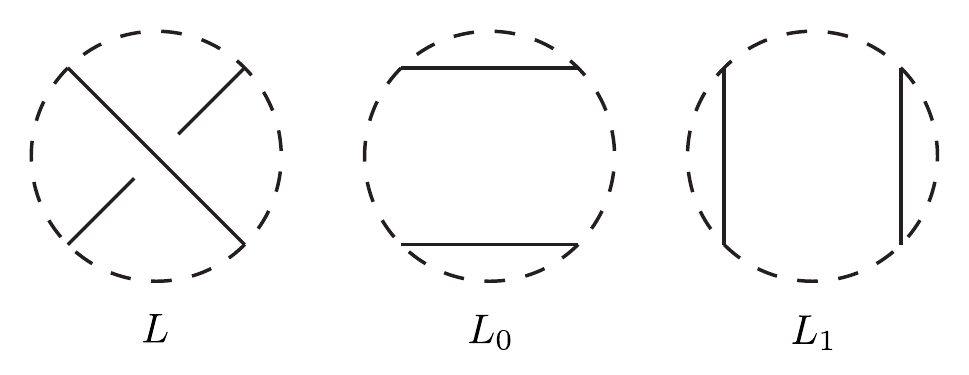}
    \caption{\textbf{Resolutions of a crossing.} The conventions agree
    with~\cite{Khovanov00:CatJones}, not~\cite{BrDCov}.}
    \label{fig:resolutions}
  \end{figure}
  The relationship between Theorem~\ref{thm:s-seq} and
  Theorem~\ref{thm:surj-tri} comes from the following observation: let
  $L$ be a link diagram, $c$ a crossing in $L$, and $L_0$ and $L_1$
  the two resolutions of $L$ at $c$, as in
  Figure~\ref{fig:resolutions}. Let $\gamma$ be the vertical arc in
  $\RR^3$ with boundary on $L$ lying above $c$. The total preimage
  $\widetilde{\gamma}$ of $\gamma$ in $\Sigma(L)$ is a circle
  $K$. There are surgery slopes $n$ and $n'$ on $\bdy \nbd(K)$ so that
  $n$ (respectively $n'$) surgery on $K$ gives $L_1$ (respectively
  $L_0$), and $(\infty,n,n')$ satisfy the conditions of
  Theorem~\ref{thm:surj-tri}. Thus, Theorem~\ref{thm:surj-tri} gives a
  long exact sequence
  \[
  \cdots\to \HFa(\Sigma(L))\to \HFa(\Sigma(L_1))\to
  \HFa(\Sigma(L_0))\to \HFa(\Sigma(L))\to\cdots.
  \]
  There is an analogous skein sequence for Khovanov homology. This skein relation does not characterize Khovanov homology but, as we will see, it almost implies the existence of a spectral sequence of the desired form.
  
  To proceed, we actually need a slight strengthening of
  Theorem~\ref{thm:surj-tri}: with notation as in that theorem, there
  is a short exact sequence of chain complexes
  \begin{equation}\label{eq:ses}
  0\to \CFa(Y_{n'}(K))\to \CFa(Y_{n''}(K))\to \CFa(Y_n(K))\to 0
  \end{equation}
  (for appropriately chosen Heegaard diagrams). Further, this surgery
  triangle is local in the following sense. Fix two disjoint knots
  $K$ and $L$ in $Y$, and framings $m,m',m''$ for $K$ and $n,n',n''$ for
  $K'$ as in Theorem~\ref{thm:surj-tri}. Let $Y_{m,n}(K\cup L)$ be
  the result of performing $m$ surgery on $K$ and $n$ surgery on $L$. Then there is a
  homotopy-commutative diagram
  \begin{equation}\label{eq:big-square}
  \xymatrix{
     & 0\ar[d] & 0\ar[d] & 0\ar[d] & \\
     0\ar[r] & \CFa(Y_{m',n'}(K\cup L))\ar[r]^{f_{00,10}}\ar[d]^{f_{00,01}} &
     \CFa(Y_{m'',n'}(K\cup L))\ar[r]\ar[d]^{f_{10,11}} & \CFa(Y_{m,n'}(K\cup
     L))\ar[r]\ar[d] & 0\\
     0\ar[r] & \CFa(Y_{m',n''}(K\cup L))\ar[r]^{f_{01,11}}\ar[d] &
     \CFa(Y_{m'',n''}(K\cup L))\ar[r]\ar[d] & \CFa(Y_{m,n''}(K\cup
     L))\ar[r]\ar[d] & 0\\
     0\ar[r] & \CFa(Y_{m',n}(K\cup L))\ar[r]\ar[d] &
     \CFa(Y_{m'',n}(K\cup L))\ar[r]\ar[d] & \CFa(Y_{m,n}(K\cup
     L))\ar[r]\ar[d] & 0\\
    & 0 & 0 & 0 &
  }
  \end{equation}
  where the rows and columns are the exact sequences of
  Equation~\ref{eq:ses}. Here, homotopy-commutative means, for
  instance, that there is a map $f_{00,11}\co \CFa(Y_{m',n'}(K\cup
  L))\to \CFa(Y_{m'',n''}(K\cup L))$ so that 
  \[
  \bdy\circ f_{00,11}+f_{00,11}\circ \bdy = f_{01,11}\circ
  f_{00,01}+f_{10,11}\circ f_{00,10}.
  \]
  The natural analogue holds for surgeries on a link of $n>2$
  components, as well.

  As an algebraic corollary, we have the fact that 
  \begin{equation}\label{eq:iterated-cone}
  \CFa(Y_{m,n}(K\cup L)) \simeq \Cone\left(
    \mathcenter{\xymatrix{\CFa(Y_{m',n'}(K\cup L))\ar[r]^{f_{00,10}}\ar[d]^{f_{00,01}} \ar[dr]^{f_{00,11}}&
     \CFa(Y_{m'',n'}(K\cup L))\ar[d]^{f_{10,11}}\\
     \CFa(Y_{m',n''}(K\cup L))\ar[r]^{f_{01,11}} &
     \CFa(Y_{m'',n''}(K\cup L))}}
  \right);
  \end{equation}
  see Exercise~\ref{exercise:cone}. Here, the $\Cone$ means that we
  take the whole diagram and view it as a complex. That is, take the
  direct sum of the complexes at the four vertices, and use the
  differentials on the complexes and maps between them to define a
  differential. For instance, if $x\in \CFa(Y_{m',n'}(K\cup L))$ then
  the differential of $x$ is given by
  $\bdy_\CF(x)+f_{00,10}(x)+f_{00,01}(x)+f_{00,11}(x)$. (If we were
  not working over $\Field$, there would be some signs.) Again, the
  analogous results hold for a link with more than two components.

  Notice that the complex in Formula~\ref{eq:iterated-cone} has an
  obvious filtration. The terms in the associated graded complex are
  $\CFa(Y_{m',n'}(K\cup L))$, $\CFa(Y_{m'',n'}(K\cup L))\oplus
  \CFa(Y_{m',n''}(K\cup L))$, and $\CFa(Y_{m'',n''}(K\cup L))$. Thus,
  there is a spectral sequence with $E^1$ term given by 
  \[
  \HFa(Y_{m',n'}(K\cup L)) \oplus\HFa(Y_{m'',n'}(K\cup L))\oplus
  \HFa(Y_{m',n''}(K\cup L))\oplus\HFa(Y_{m'',n''}(K\cup L))
  \]
  converging to $\HFa(Y_{m,n}(K\cup L))$. This is called the
  \emph{link surgery spectral sequence}.

  Returning to the branched double cover, suppose that $L$ has $k$
  crossings. Consider the link $K$ in $\Sigma(L)$ corresponding to the
  $k$ crossings, as in the first paragraph of the proof. The surgery
  spectral sequence corresponding to this link has $E^\infty$-page
  $\HFa(\Sigma(L))$. It remains to identify the $E^2$-page with
  Khovanov homology. In fact, the $E^1$-page is identified with the
  reduced Khovanov complex. At the level of vertices, notice that the
  Floer group corresponding to each vertex is the branched double
  cover of an unlink. Hence, by Exercise~\ref{exercise:unlink-HF}, if
  the unlink has $\ell$ components then this branched double cover has
  $\HFa$ given by $(\Field\oplus\Field)^{\otimes (n-1)}$, in agreement
  with the corresponding term in the reduced Khovanov
  complex. Identifying the differential on the $E^1$-page is then a
  fairly short computation; see~\cite{BrDCov}. (The arrows are exactly
  backwards from the usual Khovanov differential, which is the reason
  for the $m(L)$ in the statement of the theorem.)
\end{proof}

\subsection{Suggested exercises}
\begin{enumerate}
\item Corollary~\ref{cor:larger-surg} holds for rational surgeries, as well. Prove it.
\item\label{exercise:cone} Prove Formula~\eqref{eq:iterated-cone} (assuming Formula~\eqref{eq:big-square}).
\item\label{exercise:unlink-HF} Show that the branched double cover of an $n$-component unlink
  in $S^3$ is the connected sum of $(n-1)$ copies of $S^2\times
  S^1$. Deduce that the branched double cover has $\HFa$ given by
  $(\Field\oplus\Field)^{\otimes (n-1)}$.
\end{enumerate}

%%% Local Variables: 
%%% mode: latex
%%% TeX-master: "RLLuminySMF.tex"
%%% End: 